\documentclass[12pt,a4paper]{amsart}
\usepackage{amsmath,amsthm,amssymb,latexsym,a4wide,epic,eepic,bbm}

\newtheorem{theorem}{Theorem}
\newtheorem{lemma}[theorem]{Lemma}
\newtheorem{corollary}[theorem]{Corollary}
\newtheorem{proposition}[theorem]{Proposition}

\newtheorem{remark}[theorem]{Remark}
\usepackage[all]{xy}
\usepackage{enumerate}
\numberwithin{equation}{section}
\newcommand{\tto}{\twoheadrightarrow}

\begin{document}

\title[representations of partial Brauer algebras]{On the representation theory\\ of partial Brauer algebras}
\author{Paul Martin and Volodymyr Mazorchuk}

\begin{abstract}
In this paper we study the partial Brauer $\mathbb{C}$-algebras $\mathfrak{R}_n(\delta,\delta')$, where $n \in 
\mathbb{N}$ and $\delta,\delta'\in\mathbb{C}$. We show that these algebras are generically semisimple, construct 
the Specht modules and determine the Specht module restriction rules for the restriction $\mathfrak{R}_{n-1} 
\hookrightarrow \mathfrak{R}_n$. We also determine the corresponding decomposition matrix, and the Cartan decomposition 
matrix.
\end{abstract}
\maketitle

\section{Introduction and description of the results}\label{s1}

Let $\Bbbk$ be a commutative ring, and $\Bbbk^{\times}$ its group of units. We denote by $\mathbb{N}$ and 
$\mathbb{N}_0$ the sets of all positive integers and all nonnegative integers, respectively. For
$n\in\mathbb{N}_0$ we set $\underline{n}:=\{1,2,\dots,n\}$, $\underline{n}':=\{1',2',\dots,n'\}$ and
$\mathbf{n}:=\underline{n}\cup \underline{n}'$ (the later union is automatically disjoint).
In this paper we will use $\equiv$ to denote Morita equivalence.

For each choice of $\delta\in \Bbbk$ and $n\in\mathbb{N}_0$ 
the partition algebra $\mathfrak{P}_n(\delta)$, defined in 
\cite{Martin94}, has a $\Bbbk$-basis consisting of all 
partitions of the set 
$\mathbf{n}$.
Equivalently one may use a basis of  {\em partition diagrams} 
 representing partitions of  $\mathbf{n}$.
The composition can be summarised as 
 first juxtaposing  partition diagrams 
(as in Figure~\ref{fig1})  
and then applying a 
$\delta$-dependent {\em straightening rule}, i.e. a rule for writing any juxtaposition as a scalar multiple of some 
diagram. 
The rule involves removing isolated connected components from the 
inside of the juxtaposed diagrams; and each 
removed component contributes a factor $\delta$ to the product, 
see \cite{Martin94} (or \S2) for details. 
The partition algebra 
$\mathfrak{P}_n(\delta)$ has a well-known  
unital diagram subalgebra (i.e. a subalgebra with basis a subset of partition
diagrams) called the {\em Brauer algebra}. 
This algebra,  
denoted 
$\mathfrak{B}_n(\delta)$,
was defined in 
\cite{Brauer37}. 
The basis of $\mathfrak{B}_n(\delta)$ is formed by so-called {\em Brauer partitions} or 
{\em  
pair partitions}, that is partitions of $\mathbf{n}$ into pairs 
(in this case the partition diagrams are {\em Brauer diagrams} --- 
exemplified by Figure~\ref{fig1}(a), where left-side dots
represent elements $1,2,\dots,n$ numbered from top to bottom 
and right-side dots represent elements $1',2',\dots,n'$ 
numbered from top to bottom). 
In between $\mathfrak{P}_n(\delta)$ and  $\mathfrak{B}_n(\delta)$ 
there is another diagram subalgebra,
denoted 
$\mathcal{P}\mathfrak{B}_n(\delta)$,
with the basis consisting of all partitions of $\mathbf{n}$ into pairs and singletons (the so-called 
{\em partial Brauer partitions}).
The semigroup version of this algebra appears explicitly in \cite{Mazorchuk95}
while the algebra itself appears implicitly in \cite{GrimmWarnaar95}.

In the process of straightening the juxtaposition of two partial Brauer diagrams there are two topologically 
different connected components which can be removed, namely loops and open strings. Assigning a factor $\delta\in \Bbbk$
to each removed loop and a factor $\delta'\in \Bbbk$ to each removed open string gives a $2$-parameter version of
$\mathcal{P}\mathfrak{B}_n(\delta)$, see \cite{Mazorchuk00}, which in this paper we will denote by 
$\mathfrak{R}_n(\delta,\delta')$ to simplify notation. This is the {\em partial Brauer algebra} that is the main 
object of our study. Obviously, we have $\mathcal{P}\mathfrak{B}_n(\delta)=\mathfrak{R}_n(\delta,\delta)$.
The multiplication rule for $\mathfrak{R}_n(\delta,\delta')$ is illustrated by Figure~\ref{fig1}(b).

\begin{figure}
\special{em:linewidth 0.4pt} \unitlength 0.80mm
(a)
\begin{picture}(55.00,75.00)
\put(05.00,50.00){\makebox(0,0)[cc]{$\bullet$}}
\put(05.00,42.50){\makebox(0,0)[cc]{$\bullet$}}
\put(05.00,35.00){\makebox(0,0)[cc]{$\bullet$}}
\put(05.00,27.50){\makebox(0,0)[cc]{$\bullet$}}
\put(05.00,20.00){\makebox(0,0)[cc]{$\bullet$}}
\put(05.00,12.50){\makebox(0,0)[cc]{$\bullet$}}
\put(05.00,05.00){\makebox(0,0)[cc]{$\bullet$}}

\put(30.00,50.00){\makebox(0,0)[cc]{$\bullet$}}
\put(30.00,45.00){\makebox(0,0)[cc]{$\bullet$}}
\put(30.00,40.00){\makebox(0,0)[cc]{$\bullet$}}
\put(30.00,35.00){\makebox(0,0)[cc]{$\bullet$}}
\put(30.00,25.00){\makebox(0,0)[cc]{$\bullet$}}
\put(30.00,20.00){\makebox(0,0)[cc]{$\bullet$}}
\put(30.00,10.00){\makebox(0,0)[cc]{$\bullet$}}

\dashline{1}(05.00,55.00)(30.00,55.00)
\dashline{1}(05.00,00.00)(30.00,00.00)
\dashline{1}(05.00,00.00)(05.00,55.00)
\dashline{1}(30.00,00.00)(30.00,55.00)
\thicklines
\drawline(05.00,35.00)(30.00,25.00)
\linethickness{1pt}
\qbezier(30,50)(25,47.50)(30,45)
\qbezier(30,40)(25,37.50)(30,35)
\qbezier(30,20)(25,15)(30,10)
\qbezier(5,5)(10,08.75)(5,12.50)
\qbezier(5,20)(10,31.25)(5,42.50)
\qbezier(5,27.50)(15,38.75)(5,50)

\end{picture}
\quad (b) 
\begin{picture}(105.00,75.00)
\put(17.50,70.00){\makebox(0,0)[cc]{$\beta$}}
\put(47.50,70.00){\makebox(0,0)[cc]{$\alpha$}}
\put(87.50,70.00){\makebox(0,0)[cc]{$\beta\circ\alpha$}}
\put(05.00,60.00){\makebox(0,0)[cc]{\tiny $Z$}}
\put(30.00,60.00){\makebox(0,0)[cc]{\tiny $Y$}}
\put(35.00,60.00){\makebox(0,0)[cc]{\tiny $Y$}}
\put(60.00,60.00){\makebox(0,0)[cc]{\tiny $X$}}
\put(75.00,60.00){\makebox(0,0)[cc]{\tiny $Z$}}
\put(100.00,60.00){\makebox(0,0)[cc]{\tiny $X$}}
\put(05.00,50.00){\makebox(0,0)[cc]{$\bullet$}}
\put(05.00,45.00){\makebox(0,0)[cc]{$\bullet$}}
\put(05.00,40.00){\makebox(0,0)[cc]{$\bullet$}}
\put(05.00,35.00){\makebox(0,0)[cc]{$\bullet$}}
\put(05.00,30.00){\makebox(0,0)[cc]{$\bullet$}}
\put(05.00,25.00){\makebox(0,0)[cc]{$\bullet$}}
\put(05.00,20.00){\makebox(0,0)[cc]{$\bullet$}}
\put(05.00,15.00){\makebox(0,0)[cc]{$\bullet$}}
\put(05.00,10.00){\makebox(0,0)[cc]{$\bullet$}}
\put(05.00,05.00){\makebox(0,0)[cc]{$\bullet$}}
\put(30.00,50.00){\makebox(0,0)[cc]{$\bullet$}}
\put(30.00,45.00){\makebox(0,0)[cc]{$\bullet$}}
\put(30.00,40.00){\makebox(0,0)[cc]{$\bullet$}}
\put(30.00,35.00){\makebox(0,0)[cc]{$\bullet$}}
\put(30.00,30.00){\makebox(0,0)[cc]{$\bullet$}}
\put(30.00,25.00){\makebox(0,0)[cc]{$\bullet$}}
\put(30.00,20.00){\makebox(0,0)[cc]{$\bullet$}}
\put(30.00,15.00){\makebox(0,0)[cc]{$\bullet$}}
\put(30.00,10.00){\makebox(0,0)[cc]{$\bullet$}}
\put(30.00,05.00){\makebox(0,0)[cc]{$\bullet$}}
\put(35.00,50.00){\makebox(0,0)[cc]{$\bullet$}}
\put(35.00,45.00){\makebox(0,0)[cc]{$\bullet$}}
\put(35.00,40.00){\makebox(0,0)[cc]{$\bullet$}}
\put(35.00,35.00){\makebox(0,0)[cc]{$\bullet$}}
\put(35.00,30.00){\makebox(0,0)[cc]{$\bullet$}}
\put(35.00,25.00){\makebox(0,0)[cc]{$\bullet$}}
\put(35.00,20.00){\makebox(0,0)[cc]{$\bullet$}}
\put(35.00,15.00){\makebox(0,0)[cc]{$\bullet$}}
\put(35.00,10.00){\makebox(0,0)[cc]{$\bullet$}}
\put(35.00,05.00){\makebox(0,0)[cc]{$\bullet$}}
\put(60.00,50.00){\makebox(0,0)[cc]{$\bullet$}}
\put(60.00,45.00){\makebox(0,0)[cc]{$\bullet$}}
\put(60.00,40.00){\makebox(0,0)[cc]{$\bullet$}}
\put(60.00,35.00){\makebox(0,0)[cc]{$\bullet$}}
\put(60.00,30.00){\makebox(0,0)[cc]{$\bullet$}}
\put(60.00,25.00){\makebox(0,0)[cc]{$\bullet$}}
\put(60.00,20.00){\makebox(0,0)[cc]{$\bullet$}}
\put(60.00,15.00){\makebox(0,0)[cc]{$\bullet$}}
\put(60.00,10.00){\makebox(0,0)[cc]{$\bullet$}}
\put(60.00,05.00){\makebox(0,0)[cc]{$\bullet$}}
\put(100.00,50.00){\makebox(0,0)[cc]{$\bullet$}}
\put(100.00,45.00){\makebox(0,0)[cc]{$\bullet$}}
\put(100.00,40.00){\makebox(0,0)[cc]{$\bullet$}}
\put(100.00,35.00){\makebox(0,0)[cc]{$\bullet$}}
\put(100.00,30.00){\makebox(0,0)[cc]{$\bullet$}}
\put(100.00,25.00){\makebox(0,0)[cc]{$\bullet$}}
\put(100.00,20.00){\makebox(0,0)[cc]{$\bullet$}}
\put(100.00,15.00){\makebox(0,0)[cc]{$\bullet$}}
\put(100.00,10.00){\makebox(0,0)[cc]{$\bullet$}}
\put(100.00,05.00){\makebox(0,0)[cc]{$\bullet$}}
\put(75.00,50.00){\makebox(0,0)[cc]{$\bullet$}}
\put(75.00,45.00){\makebox(0,0)[cc]{$\bullet$}}
\put(75.00,40.00){\makebox(0,0)[cc]{$\bullet$}}
\put(75.00,35.00){\makebox(0,0)[cc]{$\bullet$}}
\put(75.00,30.00){\makebox(0,0)[cc]{$\bullet$}}
\put(75.00,25.00){\makebox(0,0)[cc]{$\bullet$}}
\put(75.00,20.00){\makebox(0,0)[cc]{$\bullet$}}
\put(75.00,15.00){\makebox(0,0)[cc]{$\bullet$}}
\put(75.00,10.00){\makebox(0,0)[cc]{$\bullet$}}
\put(75.00,05.00){\makebox(0,0)[cc]{$\bullet$}}
\dashline{1}(05.00,55.00)(30.00,55.00)
\dashline{1}(35.00,55.00)(60.00,55.00)
\dashline{1}(75.00,55.00)(100.00,55.00)
\dashline{1}(05.00,00.00)(30.00,00.00)
\dashline{1}(35.00,00.00)(60.00,00.00)
\dashline{1}(75.00,00.00)(100.00,00.00)
\dashline{1}(75.00,00.00)(75.00,55.00)
\dashline{1}(100.00,00.00)(100.00,55.00)
\dashline{1}(05.00,00.00)(05.00,55.00)
\dashline{1}(35.00,00.00)(35.00,55.00)
\dashline{1}(30.00,00.00)(30.00,55.00)
\dashline{1}(60.00,00.00)(60.00,55.00)
\dottedline[.]{1}(30.00,50.00)(35.00,50.00)
\dottedline[.]{1}(30.00,45.00)(35.00,45.00)
\dottedline[.]{1}(30.00,40.00)(35.00,40.00)
\dottedline[.]{1}(30.00,35.00)(35.00,35.00)
\dottedline[.]{1}(30.00,30.00)(35.00,30.00)
\dottedline[.]{1}(30.00,25.00)(35.00,25.00)
\dottedline[.]{1}(30.00,20.00)(35.00,20.00)
\dottedline[.]{1}(30.00,15.00)(35.00,15.00)
\dottedline[.]{1}(30.00,10.00)(35.00,10.00)
\dottedline[.]{1}(30.00,05.00)(35.00,05.00)
\thicklines
\drawline(05.00,35.00)(30.00,25.00)
\drawline(35.00,50.00)(60.00,45.00)
\drawline(35.00,45.00)(60.00,05.00)
\drawline(35.00,40.00)(60.00,25.00)
\drawline(75.00,35.00)(100.00,25.00)
\linethickness{1pt}
\qbezier(30,50)(25,47.50)(30,45)
\qbezier(30,40)(25,37.50)(30,35)
\qbezier(30,20)(25,15)(30,10)
\qbezier(5,5)(10,08.75)(5,15)
\qbezier(5,20)(10,31.25)(5,45)
\qbezier(5,30)(15,38.75)(5,50)
\qbezier(35,5)(40,10)(35,15)
\qbezier(35,10)(40,15)(35,20)
\qbezier(35,25)(40,30)(35,35)
\qbezier(60,15)(50,25)(60,35)
\qbezier(75,50)(85,38.75)(75,30)
\qbezier(75,20)(80,31.25)(75,45)
\qbezier(75,5)(80,08.75)(75,15)
\qbezier(100,15)(90,25)(100,35)
\qbezier(100,5)(80,25)(100,45)
\end{picture}
\caption{\label{fig:f1} 
(a) Brauer diagram; $\;$ 
(b) Basic composition of partial Brauer partitions
(the straightening factor here is $\delta {\delta'}^2$).}
\label{fig1}
\end{figure}
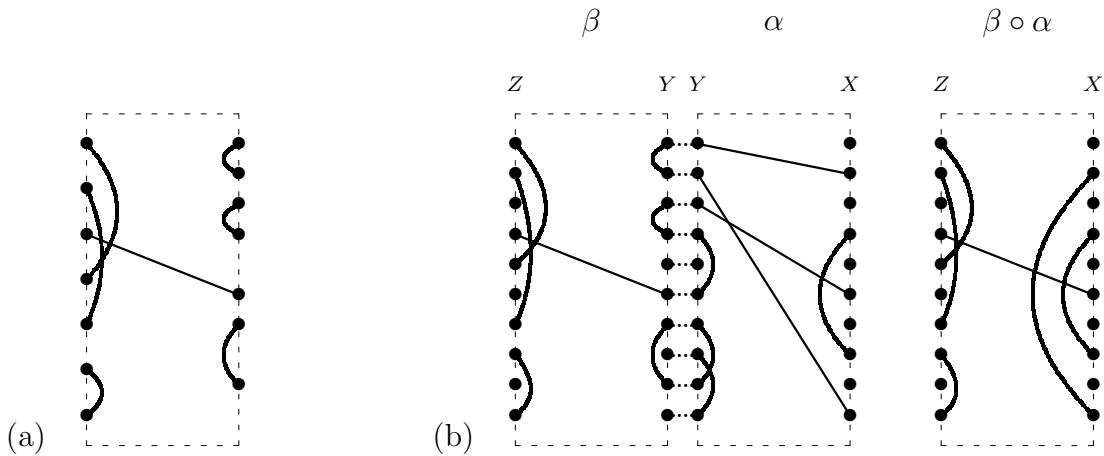

We show that the algebra $\mathfrak{R}_n(\delta,\delta')$ is generically semisimple over $\mathbb{C}$, and  construct 
{\em Specht} modules over $\mathbb{Z}[\delta,\delta']$ that pass to a full set of generic simple modules 
over $\mathbb{C}$.  
In the remaining non-semisimple cases over $\mathbb{C}$ 
the decomposition matrices for these Specht modules, and hence 
the Cartan decomposition matrices, become very complicated, however we determine them via a string of Morita 
equivalences that end up with direct sums of Brauer algebras, whose decomposition matrices are known by 
\cite{CoxDevisscherMartin0509,Martin08-9}. Our key theorem here is the following:

\begin{theorem}\label{thm1}
For $\delta' ,\delta - 1 \in \Bbbk^{\times}$  we have
$\mathfrak{R}_n (\delta,\delta' ) \; \equiv \; \mathfrak{B}_n(\delta-1) \oplus \mathfrak{B}_{n-1}(\delta-1)$. 
\end{theorem}

In Section~\ref{s2} we prove this theorem (using a direct analogue of the method used for 
the partition algebra variation treated in \cite{Martin2000}).  In Section~\ref{s3} we combine Theorem~\ref{thm1} 
with results on the representation theory of the Brauer algebra from \cite{Martin08-9} to describe the complex 
representation theory of $\mathfrak{R}_n(\delta,\delta')$.  In particular, we give an explicit construction for a 
complete set of Specht modules for $\mathfrak{R}_n(\delta,\delta')$, and show that these are images under the 
Morita equivalence of the corresponding Specht modules for the 
Brauer algebra. This means, in particular, that we 
can use the Brauer decomposition matrices from \cite{Martin08-9}. 
A mild variation of  
Theorem~\ref{thm1} holds for $\delta = 1$ 
(see Section~\ref{s4}), so that we also determine the Cartan decomposition
matrices in this case. 
As Theorem~\ref{thm1} suggests, 
provided that $\delta'$ is a unit, then it can be 
`scaled out' of representation theoretic calculations. 
In Section~\ref{s4} we also deal with
the 
$\delta' =0$ case. 

The classical Brauer algebra was defined in \cite{Brauer37} as an algebra acting on the right hand side of the 
natural generalization of the classical Schur-Weyl duality in the case of the orthogonal group.
In Section~\ref{s5} we show that the partial Brauer algebra appears on the right hand side of a natural
{\em partialization} of this Schur-Weyl duality in the spirit of Solomon's construction for the symmetric group
in \cite{So}. 

A significant feature of Theorem~\ref{thm1} 
is that it relates the partial Brauer algebra to Brauer algebras
with  a {\em different} value of the parameter $\delta$. 
Another interesting feature is
that it relates the partial Brauer algebra of rank $n$ 
to Brauer algebras of ranks $n$ and $n-1$ only
(one could  contrast this behaviour with that of certain
other partial analogues, 
%
see for example \cite{Paget06}, 
for which one gets a Morita equivalence with the direct sum
over all ranks $k=0,1,...,n$ 
of  similar algebras). 

To simplify notation we set $\mathfrak{R}_n:=\mathfrak{R}_n(\delta,\delta')$, $\mathfrak{P}_n:=\mathfrak{P}_n(\delta)$
$\mathfrak{B}_n:=\mathfrak{B}_n(\delta)$ and $\mathcal{P}\mathfrak{B}_n:=\mathcal{P}\mathfrak{B}_n(\delta)$.
\vspace{2mm}

\noindent
{\bf Acknowledgements.} An essential part of the research was done during the visit of the first author to Uppsala 
in November 2010, which was supported by the Faculty of Natural Sciences of Uppsala University.  The financial 
support and hospitality of Uppsala University are gratefully acknowledged. For the second author the research was
partially supported by the Royal Swedish Academy of Sciences and the Swedish Research Council.

\section{Preliminaries}\label{s12}

\subsection{Categorical formulation}\label{s12.1}

As a matter of expository efficiency (rather than necessity) we note the following. The partition and Brauer algebras extend in an obvious way to $\Bbbk$-linear categories \cite{Martin94,Martin08-9}, here denoted $\mathfrak{P}$ and
$\mathfrak{B}$, respectively.  The partial Brauer categories $\mathcal{P}\mathfrak{B}$ and $\mathfrak{R}$ are 
defined similarly. The category $\mathfrak{P}$ is a monoidal category with monoidal composition $a \otimes b$ defined 
as in Figure~\ref{fig3}; and an involutive antiautomorphism $\star$ (in terms of diagrams as drawn here, the $\star$  
operation is reflection in a vertical line~--- see e.g. \cite{Martin08-9}, and Figure~\ref{fig3}(b)).  It is then 
generated, as a $\Bbbk$-linear category with $\otimes $ and $\star$, by 
\[
\raisebox{5mm}{$u =$} 
\begin{picture}(63.00,35)(0,5)
\special{em:linewidth 0.4pt} \unitlength 0.80mm
\put(05.00,10.00){\makebox(0,0)[cc]{$\bullet$}}
\dashline{1}(05.00,05.00)(05.00,15.00)
\dashline{1}(25.00,05.00)(25.00,15.00)
\dashline{1}(05.00,05.00)(25.00,05.00)
\dashline{1}(05.00,15.00)(25.00,15.00)
\thicklines
\end{picture}
\raisebox{.1in}{ ,} \quad 
\raisebox{5mm}{1 = }
\begin{picture}(63.00,35.00)(0,5)
\special{em:linewidth 0.4pt} \unitlength 0.80mm
\put(05.00,10.00){\makebox(0,0)[cc]{$\bullet$}}
\put(25.00,10.00){\makebox(0,0)[cc]{$\bullet$}}
\dashline{1}(05.00,05.00)(05.00,15.00)
\dashline{1}(25.00,05.00)(25.00,15.00)
\dashline{1}(05.00,05.00)(25.00,05.00)
\dashline{1}(05.00,15.00)(25.00,15.00)
\thicklines
\drawline(05.00,10.00)(25.00,10.00)
\end{picture}
\raisebox{.1in}{ ,} \quad 
\raisebox{7mm}{$v = $}
\begin{picture}(65.00,45.00)(0,10)
\special{em:linewidth 0.4pt} \unitlength 0.80mm
\put(05.00,10.00){\makebox(0,0)[cc]{$\bullet$}}
\put(05.00,20.00){\makebox(0,0)[cc]{$\bullet$}}
\put(25.00,20.00){\makebox(0,0)[cc]{$\bullet$}}
\dashline{1}(05.00,05.00)(05.00,25.00)
\dashline{1}(25.00,05.00)(25.00,25.00)
\dashline{1}(05.00,05.00)(25.00,05.00)
\dashline{1}(05.00,25.00)(25.00,25.00)
\thicklines
\drawline(05.00,10.00)(25.00,20.00)
\drawline(05.00,20.00)(25.00,20.00)
\end{picture}
\raisebox{.1in}{ ,} \quad 
\raisebox{7mm}{$x = $}
\begin{picture}(65.00,39.00)(0,5)
\special{em:linewidth 0.4pt} \unitlength 0.80mm
\put(05.00,10.00){\makebox(0,0)[cc]{$\bullet$}}
\put(05.00,20.00){\makebox(0,0)[cc]{$\bullet$}}
\put(25.00,10.00){\makebox(0,0)[cc]{$\bullet$}}
\put(25.00,20.00){\makebox(0,0)[cc]{$\bullet$}}
\dashline{1}(05.00,05.00)(05.00,25.00)
\dashline{1}(25.00,05.00)(25.00,25.00)
\dashline{1}(05.00,05.00)(25.00,05.00)
\dashline{1}(05.00,25.00)(25.00,25.00)
\thicklines
\drawline(05.00,10.00)(25.00,20.00)
\drawline(05.00,20.00)(25.00,10.00)
\end{picture}\raisebox{.1in}{ ,}
\]
that is to say, the minimal $\Bbbk$-linear subcategory closed under $\otimes$ and $\star$ and containing these four 
elements is $\mathfrak{P}$ itself (this follows immediately from \cite[Prop.2]{Martin94}).
Similarly, $\mathfrak{B}$ is generated by 1, $x$ and 
\[
\raisebox{10mm}{$v u = $}
\begin{picture}(115.00,60.00)
\special{em:linewidth 0.4pt} \unitlength 0.80mm
\put(05.00,10.00){\makebox(0,0)[cc]{$\bullet$}}
\put(05.00,20.00){\makebox(0,0)[cc]{$\bullet$}}
\put(25.00,20.00){\makebox(0,0)[cc]{$\bullet$}}
\dashline{1}(05.00,05.00)(05.00,25.00)
\dashline{1}(25.00,05.00)(25.00,25.00)
\dashline{1}(05.00,05.00)(25.00,05.00)
\dashline{1}(05.00,25.00)(25.00,25.00)
\dashline{1}(25.00,05.00)(25.00,25.00)
\dashline{1}(45.00,05.00)(45.00,25.00)
\dashline{1}(25.00,05.00)(45.00,05.00)
\dashline{1}(25.00,25.00)(45.00,25.00)
\thicklines
\drawline(05.00,10.00)(25.00,20.00)
\drawline(05.00,20.00)(25.00,20.00)
\end{picture}
\raisebox{10mm}{=}
\begin{picture}(65.00,60.00)
\special{em:linewidth 0.4pt} \unitlength 0.80mm
\put(05.00,10.00){\makebox(0,0)[cc]{$\bullet$}}
\put(05.00,20.00){\makebox(0,0)[cc]{$\bullet$}}
\dashline{1}(05.00,05.00)(05.00,25.00)
\dashline{1}(25.00,05.00)(25.00,25.00)
\dashline{1}(05.00,05.00)(25.00,05.00)
\dashline{1}(05.00,25.00)(25.00,25.00)
%
\thicklines
\qbezier(5,10)(10,15)(5,20)
\end{picture}
\]
Finally, $\mathcal{P}\mathfrak{B}$ and $\mathfrak{R}$ are generated by $1$, $x$, $vu$ and $u$.
The algebras $\mathfrak{P}_n$, $\mathfrak{B}_n$, $\mathcal{P}\mathfrak{B}_n$ and $\mathfrak{R}_n$
are in the natural way endomorphism algebras of certain objects in the corresponding categories.


\begin{figure}
(a)
\special{em:linewidth 0.4pt} \unitlength 0.80mm
\begin{picture}(75.00,70.00)
\put(15.00,32.50){\makebox(0,0)[cc]{$\otimes$}}
\put(37.50,32.50){\makebox(0,0)[cc]{$=$}}
\put(05.00,10.00){\makebox(0,0)[cc]{$\bullet$}}
\put(05.00,20.00){\makebox(0,0)[cc]{$\bullet$}}
\put(05.00,45.00){\makebox(0,0)[cc]{$\bullet$}}
\put(05.00,50.00){\makebox(0,0)[cc]{$\bullet$}}
\put(05.00,55.00){\makebox(0,0)[cc]{$\bullet$}}
\put(05.00,60.00){\makebox(0,0)[cc]{$\bullet$}}
\put(25.00,10.00){\makebox(0,0)[cc]{$\bullet$}}
\put(25.00,15.00){\makebox(0,0)[cc]{$\bullet$}}
\put(25.00,20.00){\makebox(0,0)[cc]{$\bullet$}}
\put(25.00,45.00){\makebox(0,0)[cc]{$\bullet$}}
\put(25.00,52.50){\makebox(0,0)[cc]{$\bullet$}}
\put(25.00,60.00){\makebox(0,0)[cc]{$\bullet$}}
\put(50.00,15.00){\makebox(0,0)[cc]{$\bullet$}}
\put(50.00,25.00){\makebox(0,0)[cc]{$\bullet$}}
\put(50.00,35.00){\makebox(0,0)[cc]{$\bullet$}}
\put(50.00,40.00){\makebox(0,0)[cc]{$\bullet$}}
\put(50.00,45.00){\makebox(0,0)[cc]{$\bullet$}}
\put(50.00,50.00){\makebox(0,0)[cc]{$\bullet$}}
\put(70.00,15.00){\makebox(0,0)[cc]{$\bullet$}}
\put(70.00,20.00){\makebox(0,0)[cc]{$\bullet$}}
\put(70.00,25.00){\makebox(0,0)[cc]{$\bullet$}}
\put(70.00,35.00){\makebox(0,0)[cc]{$\bullet$}}
\put(70.00,42.50){\makebox(0,0)[cc]{$\bullet$}}
\put(70.00,50.00){\makebox(0,0)[cc]{$\bullet$}}
\dashline{1}(05.00,05.00)(05.00,25.00)
\dashline{1}(05.00,40.00)(05.00,65.00)
\dashline{1}(25.00,05.00)(25.00,25.00)
\dashline{1}(25.00,40.00)(25.00,65.00)
\dashline{1}(50.00,10.00)(50.00,55.00)
\dashline{1}(70.00,10.00)(70.00,55.00)
\dashline{1}(05.00,05.00)(25.00,05.00)
\dashline{1}(05.00,65.00)(25.00,65.00)
\dashline{1}(05.00,40.00)(25.00,40.00)
\dashline{1}(05.00,25.00)(25.00,25.00)
\dashline{1}(50.00,10.00)(70.00,10.00)
\dashline{1}(50.00,30.00)(70.00,30.00)
\dashline{1}(50.00,55.00)(70.00,55.00)
\thicklines
\drawline(05.00,55.00)(25.00,45.00)
\drawline(05.00,20.00)(25.00,15.00)
\drawline(50.00,45.00)(70.00,35.00)
\drawline(50.00,25.00)(70.00,20.00)
\qbezier(5,10)(10,15)(5,20)
\qbezier(5,50)(10,55)(5,60)
\qbezier(25,10)(20,15)(25,20)
\qbezier(25,52.50)(20,56.25)(25,60)
\qbezier(50,15)(55,20)(50,25)
\qbezier(50,40)(55,45)(50,50)
\qbezier(70,15)(65,20)(70,25)
\qbezier(70,42.50)(65,46.25)(70,50)
\end{picture}
\hspace{.541in} (b)
\begin{picture}(30.00,35)(0,.5)
\special{em:linewidth 0.4pt} \unitlength 0.80mm
\put(05.00,10.00){\makebox(0,0)[cc]{$\bullet$}}
\dashline{1}(05.00,05.00)(05.00,15.00)
\dashline{1}(25.00,05.00)(25.00,15.00)
\dashline{1}(05.00,05.00)(25.00,05.00)
\dashline{1}(05.00,15.00)(25.00,15.00)
\thicklines
\end{picture}
\raisebox{.21in}{
$\stackrel{\star}{\mapsto}$}
\begin{picture}(30.00,35)(0,.5)
\special{em:linewidth 0.4pt} \unitlength 0.80mm
\put(25.00,10.00){\makebox(0,0)[cc]{$\bullet$}}
\dashline{1}(05.00,05.00)(05.00,15.00)
\dashline{1}(25.00,05.00)(25.00,15.00)
\dashline{1}(05.00,05.00)(25.00,05.00)
\dashline{1}(05.00,15.00)(25.00,15.00)
\thicklines
\end{picture}
\caption{(a) Tensor product of partition diagrams; (b) $\star$ operation.}
\label{fig3}
\end{figure}
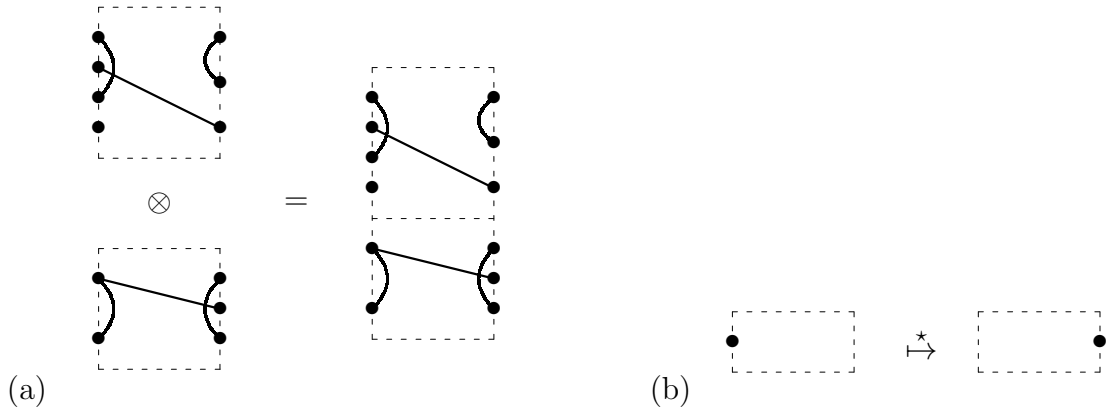

\subsection{Diagram calculus}\label{s12.2}

We denote by $\mathtt{P}_{n}$ the set of all set partitions of $\mathbf{n}$; by $\mathtt{B}_{n}$ the set of all 
Brauer partitions of $\mathbf{n}$ and by $\mathtt{R}_{n}$ the set of all partial Brauer partitions of $\mathbf{n}$.
We identify partitions with diagrams as explained in the introduction. 

For  $d,d'\in \mathtt{P}_{n}$ we denote by $d\circ d'$ the element of $\mathtt{P}_{n}$ obtained by 
juxtaposing $d$ (on the left) and $d'$ (on the right) and then applying the straitening rule with $\delta=1$.

In a partition diagram, a part with vertices on both sides of the diagram is called a {\em propagating part} 
or a propagating line. The number $\#^{p}(d)$ of propagating lines is called the {\em propagating number}
(in the literature this is sometimes also known as {\em rank}). All diagrams of the maximal rank $n$ form
a copy of the symmetric group $S_n$ in $\mathfrak{R}_n$. The number $\#^{s}(d)$ of singleton parts in
$d$ is called the {\em defect} of $d$. When two diagrams $d,d'$ are concatenated in 
composition, we call the `middle' layer formed (before straightening) the {\em equator}. We write $d|d'$ for 
the concatenated unstraightened `diagram'. 

The following elementary exercise will illustrate the diagram calculus machinery
(of juxtaposition and straightening) in the partial Brauer case, and also be useful later on. Define 
$\mathtt{R}_{n}^{(l)}$ as the set of partial Brauer partitions with $l$ singletons. For $d \in \mathtt{R}_{n}^{(l)}$ 
and  $d' \in \mathtt{R}_{n}^{(l')}$, the singletons from $d$ and $d'$ appear in  $d|d'$ in three possible ways: 
\begin{enumerate}[(I)]
\item\label{en1.1} in the exterior (becoming singletons of $dd'$);
\item\label{en1.2} as endpoints of open strings in the equator, i.e. in pairs connected by a (possibly zero length)
chain of pair parts;
\item\label{en1.3} as endpoints of chains terminating in the exterior.
\end{enumerate}
Let $2m$ be the number of singletons in $d|d'$ of type \eqref{en1.2}. Then it is easy to show that for 
some $k\in\mathbb{N}_0$ and $\hat{d}\in\mathtt{R}_{n}^{(l+l'-2m)}$ we have $dd'=\delta^{k}{\delta'}^{m}\hat{d}$.

One should keep in mind that every partial Brauer  diagram $d$ encodes a partition. Thus the assertion $\{i,j\} \in d$
for $i,j\in\mathbf{n}$ means that there is a line between vertices $i$ and $j$ in $d$.

Diagram calculus can also be extended to other elements of our
algebras. 
Assume $\delta'\in\Bbbk^{\times}$
and consider the element $\hat{u}\in\mathfrak{R}_1$ defined as
follows: 
$$
\hat{u} \; := \; 1-\frac{1}{\delta'}u\otimes u^{\star} 
\; = \; \{\{ 1,1' \}\} -\frac{1}{\delta'} \{\{ 1 \},\{ 1' \}\}.
$$
We denote this combination by $\{ \{ 1,1'  \}_{\blacksquare} \}$,
by direct analogy with \cite[\S3.1]{MartinGreenParker07}.
Similarly 
we will depict the element $\hat{u}$ as follows, by decorating the propagating line of the diagram of $1$ with a box:
\begin{displaymath}
\raisebox{7mm}{$\hat{u}:=$}
\begin{picture}(30.00,35)(0,.5)
\special{em:linewidth 0.4pt} \unitlength 0.80mm
\put(05.00,10.00){\makebox(0,0)[cc]{$\bullet$}}
\put(25.00,10.00){\makebox(0,0)[cc]{$\bullet$}}
\put(15.00,10.00){\makebox(0,0)[cc]{$\blacksquare$}}
\drawline(05.00,10.00)(25.00,10.00)
\dashline{1}(05.00,05.00)(05.00,15.00)
\dashline{1}(25.00,05.00)(25.00,15.00)
\dashline{1}(05.00,05.00)(25.00,05.00)
\dashline{1}(05.00,15.00)(25.00,15.00)
\thicklines
\end{picture} \quad\quad\quad\quad
\raisebox{7mm}{$=$}
\quad
\begin{picture}(30.00,35)(0,.5)
\special{em:linewidth 0.4pt} \unitlength 0.80mm
\put(05.00,10.00){\makebox(0,0)[cc]{$\bullet$}}
\put(25.00,10.00){\makebox(0,0)[cc]{$\bullet$}}
\drawline(05.00,10.00)(25.00,10.00)
\dashline{1}(05.00,05.00)(05.00,15.00)
\dashline{1}(25.00,05.00)(25.00,15.00)
\dashline{1}(05.00,05.00)(25.00,05.00)
\dashline{1}(05.00,15.00)(25.00,15.00)
\thicklines
\end{picture} 
\quad\quad\quad\quad
\raisebox{7mm}{$-\,\,\,\frac{1}{\delta'}$}
\quad
\begin{picture}(30.00,35)(0,.5)
\special{em:linewidth 0.4pt} \unitlength 0.80mm
\put(05.00,10.00){\makebox(0,0)[cc]{$\bullet$}}
\put(25.00,10.00){\makebox(0,0)[cc]{$\bullet$}}
\dashline{1}(05.00,05.00)(05.00,15.00)
\dashline{1}(25.00,05.00)(25.00,15.00)
\dashline{1}(05.00,05.00)(25.00,05.00)
\dashline{1}(05.00,15.00)(25.00,15.00)
\thicklines
\end{picture} 
\end{displaymath}
It is straightforward to check that $\hat{u}$ is an idempotent.  

We set $1_n:=1\otimes 1\otimes\cdots\otimes 1$ ($n$ factors). 
This is the identity element of $\mathfrak{R}_n$.
For $n\in\mathbb{N}$ and $i\in\underline{n}$ set 
\begin{displaymath}
\hat{u}_{n,i}:=\underbrace{1\otimes 1\otimes\cdots \otimes 1}_{i-1 \text{ factor}}\otimes \hat{u} \otimes 
\underbrace{1\otimes 1\otimes\cdots \otimes 1}_{n-i \text{ factors}}\in \mathfrak{R}_{n}.
\end{displaymath}
Each $\hat{u}_{n,i}$ is an idempotent. Define 
$\hat{u}_{n}:=\hat{u}_{n,1}\hat{u}_{n,2}\cdots\hat{u}_{n,n}$. Note that the factors of this product commute. 
It follows that $\hat{u}_{n}$ is an idempotent, moreover,
\begin{equation}\label{eq2}
\hat{u}_{n}\hat{u}_{n,i}=\hat{u}_{n,i}\hat{u}_{n}=\hat{u}_{n}. 
\end{equation}
Similarly we define the elements
\begin{displaymath}
{u}_{n,i}:=\underbrace{1\otimes 1\otimes\cdots \otimes 1}_{i-1 \text{ factor}}\otimes {u}\otimes u^{\star} \otimes 
\underbrace{1\otimes 1\otimes\cdots \otimes 1}_{n-i \text{ factors}}\in \mathfrak{R}_{n}
\end{displaymath}
and ${u}_{n}:={u}_{n,1}{u}_{n,2}\cdots{u}_{n,n}$. The elements ${u}_{n,i}$ and ${u}_{n}$ are idempotent
provided that $\delta'=1$. 

Let $d\in\mathtt{R}_n$ and assume that $\{i,j\}\in d$ for some
$i,j\in\mathbf{n}$. It is straightforward to check that:
\begin{equation}\label{eq1}
\begin{array}{ccc}
\hat{u}_{n,i}d= \hat{u}_{n,j}d, & \text{if} & i,j\in\underline{n}; \\
d\hat{u}_{n,i}= d\hat{u}_{n,j}, & \text{if} & i,j\in\underline{n}';\\
\hat{u}_{n,i}d= d\hat{u}_{n,j}, & \text{if} & i\in\underline{n},j\in\underline{n}'; \\
d\hat{u}_{n,i}= \hat{u}_{n,j}d, & \text{if} & j\in\underline{n},i\in\underline{n}'
\end{array}
\end{equation}
(for $i \in \underline{n}$ set $\hat{u}_{n,i'} = \hat{u}_{n,i}$).
In each case, the corresponding element $\hat{u}_{n,i}d\in \mathfrak{R}_n$, for $i\in \underline{n}$, 
or $d\hat{u}_{n,j}\in \mathfrak{R}_n$, for $j\in \underline{n}'$,  will be depicted by the same diagram as $d$ 
but with the part $\{i,j\}$ decorated by a box just like in the case of the element $\hat{u}$ defined above. 

Define the set $\hat{\mathtt{R}}_n$ of 
{\em decorated (partition) diagrams} as follows: a decorated diagram is a diagram from ${\mathtt{R}}_n$ in which, 
additionally, some lines are decorated by a box (the position of the box on the line does not matter). 
We identify $\hat{\mathtt{R}}_n$ with certain elements from $\mathfrak{R}_n$ recursively as follows: first of 
all ${\mathtt{R}}_n\subset \hat{\mathtt{R}}_n$ with the natural identification. Then, given $d\in \hat{\mathtt{R}}_n$
with an already fixed identification, and an undecorated part $\{i,j\}\in d$, the decoration of this part gives
a new decorated diagram which we identify with the element $\hat{u}_{n,i}d$ (in the case $i\in \underline{n}$) 
or $d\hat{u}_{n,i}$ (in the case $i\in \underline{n}'$).  From \eqref{eq1} it follows that this does not depend on 
the choice of a representative in $\{i,j\}$. We leave it to the reader to check that the decorated diagram (as an 
element of $\mathfrak{R}_n$) does not depend on the order in which its lines are decorated. 

For $d\in {\mathtt{R}}_n$ and any set $X$ of pair parts let $d_X$ denote the decorated diagram obtained from $d$ 
by decorating all parts from $X$, and $d^X$ denote the undecorated diagram obtained from $d$ by splitting all parts 
in $X$ into singletons. Then one checks that the above identification can be reformulated as follows:
\begin{displaymath}
d_X=\sum_{Y\subset X} \frac{1}{{\delta'}^{|Y|}}d^Y.
\end{displaymath}
For $d\in {\mathtt{R}}_n$ we denote by $\langle d\rangle$ the element from $\hat{\mathtt{R}}_n$ obtained from
$d$ by decorating all pair parts. We have $\langle 1_n\rangle=\hat{u}_n$. For $X\subset\underline{n}$ define
\begin{displaymath}
\hat{u}_X:=\prod_{i\in X}\hat{u}_{n,i}  
\end{displaymath}
and note that the factors of this product commute. For $d\in\mathtt{R}_n$ let $l(d)\subset\underline{n}$ consist
of all $i$ such that $\{i\}$ is not a singleton in $d$. Similarly, let $r(d)\subset\underline{n}$ consist
of all $i$ such that $\{i'\}$ is not a singleton in $d$. Then from the above we have
$\langle d\rangle=\hat{u}_{l(d)}d\hat{u}_{r(d)}$.

\begin{proposition}[Calculus of decorated diagrams]\label{prop1}{\tiny \hspace{1mm}}

\begin{enumerate}[$($i$)$]
\item\label{prop1.1} The elements $\langle d\rangle$, $d\in {\mathtt{R}}_n$, form a basis of $\mathfrak{R}_n$.
\item\label{prop1.2} For $d,d'\in {\mathtt{R}}_n$ we have 
\begin{displaymath}
\langle d\rangle \langle d'\rangle=
\begin{cases}
0, & \text{if a singleton meets a pair part at the equator of $d|d'$};\\
(\delta-1)^k{\delta'}^{k'}\langle d\circ d'\rangle, & \text{otherwise};
\end{cases}
\end{displaymath}
where $k$ and $k'$ are the numbers of closed loops and open strings, respectively, appearing in the
straightening procedure for $d\vert d'$.
\end{enumerate}
\end{proposition}

\begin{proof}
Choose some linear order $<$ on ${\mathtt{R}}_n$ such that for any $d,d'\in {\mathtt{R}}_n$ satisfying
$\#^{s}(d)<\#^{s}(d')$ we have $d<d'$. Then, with respect to this order, the transformation matrix from the basis
$\{d:d\in {\mathtt{R}}_n\}$ to the set $\{\langle d\rangle:d\in {\mathtt{R}}_n\}$ is a square upper triangular matrix
with $1$'s on the diagonal. This implies claim \eqref{prop1.1}.
 
To prove claim  \eqref{prop1.2} we first observe that $\langle d\rangle \langle d'\rangle=
\hat{u}_{l(d)}d\hat{u}_{r(d)}\hat{u}_{l(d')}d'\hat{u}_{r(d')}$. Assume that a singleton part meets a pair part at 
the equator of $d|d'$, say in position $i$. We assume that the singleton part is from $d$ (the other case is 
obtained by applying $\star$). Then, using \eqref{eq2}, we have 
\begin{multline*}
\hat{u}_{l(d)}d\hat{u}_{r(d)}\hat{u}_{l(d')}d'\hat{u}_{r(d')}=
\hat{u}_{l(d)}d\hat{u}_{n,i}\hat{u}_{r(d)}\hat{u}_{l(d')}d'\hat{u}_{r(d')}=\\=
\hat{u}_{l(d)}d\hat{u}_{r(d)}\hat{u}_{l(d')}d'\hat{u}_{r(d')}
-\frac{1}{\delta'}\hat{u}_{l(d)}d{u}_{n,i}\hat{u}_{r(d)}\hat{u}_{l(d')}d'\hat{u}_{r(d')}.
\end{multline*}
Every diagram appearing (before straightening) in the second summand of the last formula has,
compared to the first summand, an extra singleton pair (forming one open string) at level $i$. 
This implies that 
\begin{multline*}
\hat{u}_{l(d)}d\hat{u}_{r(d)}\hat{u}_{l(d')}d'\hat{u}_{r(d')}
-\frac{1}{\delta'}\hat{u}_{l(d)}d{u}_{n,i}\hat{u}_{r(d)}\hat{u}_{l(d')}d'\hat{u}_{r(d')}=\\=
\hat{u}_{l(d)}d\hat{u}_{r(d)}\hat{u}_{l(d')}d'\hat{u}_{r(d')}
-\frac{\delta'}{\delta'}\hat{u}_{l(d)}d\hat{u}_{r(d)}\hat{u}_{l(d')}d'\hat{u}_{r(d')}=0,
\end{multline*}
which established the first line of claim \eqref{prop1.2}. Pictorially, this can be drawn as follows
(disregarding all parts of all diagrams, which are not directly involved in the computation):
\vspace{2mm}

\begin{displaymath}
\begin{picture}(30,20)
\special{em:linewidth 0.4pt} \unitlength 0.80mm
\put(14.00,05.00){\makebox(0,0)[cc]{$\bullet$}}
\put(14.00,15.00){\makebox(0,0)[cc]{$\bullet$}}
\put(09.00,10.00){\makebox(0,0)[cc]{$\blacksquare$}}
\put(16.00,15.00){\makebox(0,0)[cc]{$\bullet$}}
\dashline{1}(01.00,01.00)(01.00,19.00)
\dashline{1}(01.00,19.00)(14.00,19.00)
\dashline{1}(14.00,19.00)(14.00,01.00)
\dashline{1}(01.00,01.00)(14.00,01.00)
\dashline{1}(16.00,01.00)(16.00,19.00)
\dashline{1}(16.00,19.00)(29.00,19.00)
\dashline{1}(29.00,19.00)(29.00,01.00)
\dashline{1}(16.00,01.00)(29.00,01.00)
\thicklines
\qbezier(14,05)(3,10)(14,15)
\end{picture} 
\quad\quad\quad\quad
\raisebox{7mm}{$=$}
\quad
\begin{picture}(30,20)
\special{em:linewidth 0.4pt} \unitlength 0.80mm
\put(14.00,05.00){\makebox(0,0)[cc]{$\bullet$}}
\put(14.00,15.00){\makebox(0,0)[cc]{$\bullet$}}
\put(16.00,15.00){\makebox(0,0)[cc]{$\bullet$}}
\dashline{1}(01.00,01.00)(01.00,19.00)
\dashline{1}(01.00,19.00)(14.00,19.00)
\dashline{1}(14.00,19.00)(14.00,01.00)
\dashline{1}(01.00,01.00)(14.00,01.00)
\dashline{1}(16.00,01.00)(16.00,19.00)
\dashline{1}(16.00,19.00)(29.00,19.00)
\dashline{1}(29.00,19.00)(29.00,01.00)
\dashline{1}(16.00,01.00)(29.00,01.00)
\thicklines
\qbezier(14,05)(3,10)(14,15)
\end{picture} 
\quad\quad\quad\quad
\raisebox{7mm}{$-\frac{1}{\delta'}$} 
\begin{picture}(30,20)
\special{em:linewidth 0.4pt} \unitlength 0.80mm
\put(14.00,05.00){\makebox(0,0)[cc]{$\bullet$}}
\put(14.00,15.00){\makebox(0,0)[cc]{$\bullet$}}
\put(16.00,15.00){\makebox(0,0)[cc]{$\bullet$}}
\dashline{1}(01.00,01.00)(01.00,19.00)
\dashline{1}(01.00,19.00)(14.00,19.00)
\dashline{1}(14.00,19.00)(14.00,01.00)
\dashline{1}(01.00,01.00)(14.00,01.00)
\dashline{1}(16.00,01.00)(16.00,19.00)
\dashline{1}(16.00,19.00)(29.00,19.00)
\dashline{1}(29.00,19.00)(29.00,01.00)
\dashline{1}(16.00,01.00)(29.00,01.00)
\thicklines
\end{picture} 
\quad\quad\quad\quad
\raisebox{7mm}{$=\left(1-\frac{\delta'}{\delta'}\right)$}
\begin{picture}(15,20)
\special{em:linewidth 0.4pt} \unitlength 0.80mm
\dashline{1}(01.00,01.00)(01.00,19.00)
\dashline{1}(01.00,19.00)(14.00,19.00)
\dashline{1}(14.00,19.00)(14.00,01.00)
\dashline{1}(01.00,01.00)(14.00,01.00)
\thicklines
\end{picture} 
\quad\quad
\raisebox{7mm}{$=0$.}
\end{displaymath}

To prove second line of claim \eqref{prop1.2} we observe that the definition of $\hat{u}$ implies that the removal of
a decorated loop splits into a linear combination, in which in the first summand we remove a usual loop and
in the second summand (which has coefficient $\frac{1}{\delta'}$) we remove an open string (when all parts not in
the loop are disregarded). Therefore the total coefficient after removal of the oriented loop is 
$\delta-\frac{1}{\delta'}\delta'=\delta-1$. This can be depicted as follows:

\begin{displaymath}
\begin{picture}(55.00,50)(15,15)
\special{em:linewidth 0.4pt} \unitlength 0.80mm
\put(10.00,15.00){\makebox(0,0)[cc]{$\blacksquare$}}
\dashline{1}(05.00,05.00)(05.00,25.00)
\dashline{1}(25.00,05.00)(25.00,25.00)
\dashline{1}(05.00,05.00)(25.00,05.00)
\dashline{1}(05.00,25.00)(25.00,25.00)
\thicklines
\qbezier(15,10)(30,15)(15,20)
\qbezier(15,10)(3,15)(15,20)
\end{picture}
\raisebox{.21in}{$=$ }
\begin{picture}(59.00,50)(5,15)
\special{em:linewidth 0.4pt} \unitlength 0.80mm
\dashline{1}(05.00,05.00)(05.00,25.00)
\dashline{1}(25.00,05.00)(25.00,25.00)
\dashline{1}(05.00,05.00)(25.00,05.00)
\dashline{1}(05.00,25.00)(25.00,25.00)
\thicklines
\qbezier(15,10)(30,15)(15,20)
\qbezier(15,10)(3,15)(15,20)
\end{picture}
\raisebox{.21in}{ $-\frac{1}{\delta'}$ }
\begin{picture}(55.00,50)(10,15)
\special{em:linewidth 0.4pt} \unitlength 0.80mm
\dashline{1}(05.00,05.00)(05.00,25.00)
\dashline{1}(25.00,05.00)(25.00,25.00)
\dashline{1}(05.00,05.00)(25.00,05.00)
\dashline{1}(05.00,25.00)(25.00,25.00)
\thicklines
\qbezier(15,10)(30,15)(15,20)
\qbezier(15,10)(12,10)(11,13)
\qbezier(15,20)(11,19)(11,17)
\end{picture}
\raisebox{.21in}{ $= (\delta - \frac{1}{\delta'}\delta')$ }
\begin{picture}(55.00,50)(10,15)
\special{em:linewidth 0.4pt} \unitlength 0.80mm
\dashline{1}(05.00,05.00)(05.00,25.00)
\dashline{1}(25.00,05.00)(25.00,25.00)
\dashline{1}(05.00,05.00)(25.00,05.00)
\dashline{1}(05.00,25.00)(25.00,25.00)
\thicklines
\end{picture}
\end{displaymath}
This completes the proof.
\end{proof}

\section{Morita equivalence theorem}\label{s2}

In this section we prove Theorem~\ref{thm1}. We split the proof in a number of steps. Until the end of the
section we assume $\delta',\delta-1\in\Bbbk^{\times}$.

\subsection{Parity decomposition of $\mathfrak{R}_n$}\label{s2.1}

Denote by $\mathtt{R}_n^0$ the set of all $d\in \mathtt{R}_n$ such that $n-\#^p(d)$ is even  (such diagrams will be 
called {\em even}) and set $\mathtt{R}_n^1:=\mathtt{R}_n\setminus\mathtt{R}_n^0$ (diagrams from $\mathtt{R}_n^1$ will 
be called {\em odd}). For $i\in\{0,1\}$ denote by $\mathfrak{R}_n^i$ the linear span of $\langle d\rangle$, 
$d\in \mathtt{R}_n^i$. For $X\subset \underline{n}$ define 
\begin{displaymath}
u_X:=\prod_{i\in X}u_{n,i}
\end{displaymath}
(note that the factors in the product commute) and set $\omega_X:=\langle u_X\rangle$. In particular, we 
have $\omega_{\varnothing}=1_n$ and $\omega_{\varnothing}=\hat{u}_n$. Let $\mathbf{e}(\underline{n})$ and 
$\mathbf{o}(\underline{n})$  denote the sets of all subsets $X$ of $\underline{n}$ such that $|X|$ is even or 
odd, respectively.

\begin{proposition}\label{prop2}
Assume $\delta'\in\Bbbk^{\times}$, then we have the following:
\begin{enumerate}[$($i$)$]
\item\label{prop2.1} We have a decomposition  $\mathfrak{R}_n=\mathfrak{R}_n^0\oplus \mathfrak{R}_n^1$ 
into a direct sum of two subalgebras.
\item\label{prop2.2} The elements 
\begin{displaymath}
1_e:=\displaystyle\sum_{X\in \mathbf{e}(\underline{n})}\frac{1}{{\delta'}^{|X|}}\omega_X\quad\text{ and }\quad
1_o:=\displaystyle\sum_{X\in \mathbf{o}(\underline{n})}\frac{1}{{\delta'}^{|X|}}\omega_X 
\end{displaymath}
are the identities of the algebras $\mathfrak{R}_n^0$ and $\mathfrak{R}_n^1$, respectively.
\end{enumerate}
\end{proposition}

\begin{proof}
Notice that a partial Brauer diagram with $l$ propagating lines and $n-l$ odd has an odd number of singletons on 
both the left and the right edges. From Proposition~\ref{prop1}\eqref{prop1.2} it follows that the product 
$\langle d\rangle\langle d'\rangle$, for $d,d'\in \mathtt{R}_n$, is zero unless $n-\#^p(d)$ and $n-\#^p(d')$
are of the same parity. Hence, to complete the proof of claim \eqref{prop2.1} we just need to show that for all $d,d'$ 
with nonzero $\langle d\rangle\langle d'\rangle$ the parity of $dd'$ is the same as the parity of $d$ (or $d'$).

By Proposition~\ref{prop1}\eqref{prop1.2}, $\langle d\rangle\langle d'\rangle\neq 0$ implies that the singletons in 
the equator match up. Consider the contribution of a propagating line in $d$, say, to $dd'$.  This either forms part 
of a propagating line in $dd'$ (possibly via a chain of arcs in the equator) in combination with precisely one 
propagating line from $d'$; or else is joined, via a chain  of arcs in the equator, to another propagating line in 
$d$, and hence does not contribute to a propagating line in $dd'$. Claim \eqref{prop2.1} follows.

Claim \eqref{prop2.2} follows from Proposition~\ref{prop1} by a direct calculation.
\end{proof}

Note that from Proposition~\ref{prop2} it follows that $1_n=1_e+1_o$. We will call 
$\mathfrak{R}_n^0$ and $\mathfrak{R}_n^1$ the {\em even} and the {\em odd} summands, respectively.
After Proposition~\ref{prop2} to prove Theorem~\ref{thm1} it is left to analyze the two summands
$\mathfrak{R}_n^0$ and $\mathfrak{R}_n^1$.

\subsection{The even summand}\label{s2.2}

\begin{proposition}\label{prop3}
Assume $\delta'\in\Bbbk^{\times}$, then we have the following:
\begin{enumerate}[$($i$)$]
\item\label{prop3.1} For $d\in\mathtt{R}_n$ we have $\hat{u}_n d\hat{u}_n\neq 0$ if and only if
$d\in \mathtt{B}_n$.
\item\label{prop3.2} The set $\{\langle d\rangle:d\in \mathtt{B}_n\}$ is a basis of $\hat{u}_n\mathfrak{R}_n\hat{u}_n$.
\item\label{prop3.3} The $\Bbbk$-linear map $\gamma:\hat{u}_n\mathfrak{R}_n\hat{u}_n\to \mathfrak{B}_n(\delta-1)$, 
given by $\langle d\rangle\mapsto d$, $d\in \mathtt{B}_n$, is an isomorphism of $\Bbbk$-algebras.
\end{enumerate}
\end{proposition}

\begin{proof}
As $\hat{u}_n$ is an idempotent, we have 
$\langle d\rangle=\hat{u}_n d\hat{u}_n=\hat{u}_n\hat{u}_n d\hat{u}_n\hat{u}_n=\hat{u}_n\langle d\rangle\hat{u}_n$. 
Now claim \eqref{prop3.1} follows directly from Proposition~\ref{prop1}\eqref{prop1.2}. Further,
claim  \eqref{prop3.2} follows from \eqref{prop3.1} and
Proposition~\ref{prop1}\eqref{prop1.1}. Finally, claim  \eqref{prop3.3} follows from \eqref{prop3.2}
and Proposition~\ref{prop1}\eqref{prop1.2}.
\end{proof}

From Proposition~\ref{prop2} we have the following decomposition:
\begin{equation}\label{eq652}
\mathfrak{R}_n\text{-}\mathrm{mod}\cong\mathfrak{R}_n^0\text{-}\mathrm{mod}\oplus\mathfrak{R}_n^1\text{-}\mathrm{mod}.
\end{equation}
The first direct summand of this decomposition is described by the following statement:

\begin{proposition}\label{prop4}
Assume $\delta',(\delta-1)\in\Bbbk^{\times}$, then the functor 
\begin{equation}\label{eqf1}
\mathrm{F}:=
\hat{u}_n\mathfrak{R}_n\otimes_{\mathfrak{R}_n}{}_{-}: \mathfrak{R}_n\text{-}\mathrm{mod}
\to \hat{u}_n\mathfrak{R}_n\hat{u}_n\text{-}\mathrm{mod}
\end{equation}
(or, equivalently, $\mathrm{F}\cong\mathrm{Hom}_{\mathfrak{R}_n}(\mathfrak{R}_n\hat{u}_n,{}_{-})$)
induces an equivalence $\mathfrak{R}_n^0\text{-}\mathrm{mod}\cong \mathfrak{B}_n(\delta-1)\text{-}\mathrm{mod}$.
\end{proposition}

\begin{proof}
By \cite[Section~5]{Au}, $\mathrm{F}$ induces an equivalence between the full subcategory $\mathcal{X}$ of 
$\mathfrak{R}_n\text{-}\mathrm{mod}$, consisting of modules $M$ having a presentation $X_1\to X_0\to M\to 0$
with $X_1,X_0\in \mathrm{add}(\mathfrak{R}_n\hat{u}_n)$, where $\mathrm{add}(\mathfrak{R}_n\hat{u}_n)$ stands
for the additive closure of $\mathfrak{R}_n\hat{u}_n$, and $\hat{u}_n\mathfrak{R}_n\hat{u}_n\text{-}\mathrm{mod}$.
By Proposition~\ref{prop3}, the category $\hat{u}_n\mathfrak{R}_n\hat{u}_n\text{-}\mathrm{mod}$ is equivalent to 
$\mathfrak{B}_n(\delta-1)\text{-}\mathrm{mod}$. So, to complete the proof we have just to show that 
$\mathcal{X}\cong \mathfrak{R}_n^0\text{-}\mathrm{mod}$. For this it is enough to show that the trace 
$\mathfrak{R}_n\hat{u}_n\mathfrak{R}_n$ of $\mathfrak{R}_n\hat{u}_n$ in $\mathfrak{R}_n$ coincides with 
$\mathfrak{R}_n^0$. As $\hat{u}_n\in \mathfrak{R}_n^0$, we have $\mathfrak{R}_n\hat{u}_n\mathfrak{R}_n\subset
\mathfrak{R}_n^0$. 

As $\mathfrak{R}_n\hat{u}_n\mathfrak{R}_n$ is an ideal, it is left to show that it contains $1_e$. For this
it is enough to show that $\mathfrak{R}_n\hat{u}_n\mathfrak{R}_n$ contains $\omega_X$ for each 
$X\subset\underline{n}$ such that $|X|$ is even (see Proposition~\ref{prop2}\eqref{prop2.2}). 
As $\mathfrak{R}_n\hat{u}_n\mathfrak{R}_n$ is stable under 
multiplication with the elements of the symmetric group, it is enough to show that 
$\mathfrak{R}_n\hat{u}_n\mathfrak{R}_n$ contains $\omega_{\{1,2,\dots,2k\}}$ for each $k$ such that 
$2k\leq n$. Consider the element $x:=(u\otimes u\otimes (vu)^{\star})^{\otimes k}\otimes 1^{\otimes (n-2k)}$.
Then a direct calculation using Proposition~\ref{prop1} shows that
\begin{displaymath}
(\delta-1)^k\omega_{\{1,2,\dots,2k\}}= \langle x\rangle\hat{u}_n\langle x^{\star}\rangle
\end{displaymath}
(this is  illustrated by Figure~\ref{fig11}) and the claim follows inverting $(\delta-1)^k$.
\end{proof}

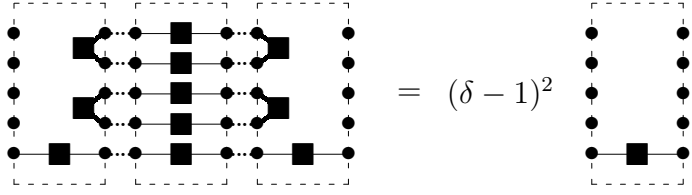
\begin{figure}
\special{em:linewidth 0.4pt} \unitlength 0.80mm
\begin{picture}(125,40)
\put(05,10){\makebox(0,0)[cc]{$\bullet$}}
\put(05,15){\makebox(0,0)[cc]{$\bullet$}}
\put(05,20){\makebox(0,0)[cc]{$\bullet$}}
\put(05,25){\makebox(0,0)[cc]{$\bullet$}}
\put(05,30){\makebox(0,0)[cc]{$\bullet$}}
\put(20,10){\makebox(0,0)[cc]{$\bullet$}}
\put(20,15){\makebox(0,0)[cc]{$\bullet$}}
\put(20,20){\makebox(0,0)[cc]{$\bullet$}}
\put(20,25){\makebox(0,0)[cc]{$\bullet$}}
\put(20,30){\makebox(0,0)[cc]{$\bullet$}}
\put(25,10){\makebox(0,0)[cc]{$\bullet$}}
\put(25,15){\makebox(0,0)[cc]{$\bullet$}}
\put(25,20){\makebox(0,0)[cc]{$\bullet$}}
\put(25,25){\makebox(0,0)[cc]{$\bullet$}}
\put(25,30){\makebox(0,0)[cc]{$\bullet$}}
\put(40,10){\makebox(0,0)[cc]{$\bullet$}}
\put(40,15){\makebox(0,0)[cc]{$\bullet$}}
\put(40,20){\makebox(0,0)[cc]{$\bullet$}}
\put(40,25){\makebox(0,0)[cc]{$\bullet$}}
\put(40,30){\makebox(0,0)[cc]{$\bullet$}}
\put(45,10){\makebox(0,0)[cc]{$\bullet$}}
\put(45,15){\makebox(0,0)[cc]{$\bullet$}}
\put(45,20){\makebox(0,0)[cc]{$\bullet$}}
\put(45,25){\makebox(0,0)[cc]{$\bullet$}}
\put(45,30){\makebox(0,0)[cc]{$\bullet$}}
\put(60,10){\makebox(0,0)[cc]{$\bullet$}}
\put(60,15){\makebox(0,0)[cc]{$\bullet$}}
\put(60,20){\makebox(0,0)[cc]{$\bullet$}}
\put(60,25){\makebox(0,0)[cc]{$\bullet$}}
\put(60,30){\makebox(0,0)[cc]{$\bullet$}}
\put(100,10){\makebox(0,0)[cc]{$\bullet$}}
\put(100,15){\makebox(0,0)[cc]{$\bullet$}}
\put(100,20){\makebox(0,0)[cc]{$\bullet$}}
\put(100,25){\makebox(0,0)[cc]{$\bullet$}}
\put(100,30){\makebox(0,0)[cc]{$\bullet$}}
\put(115,10){\makebox(0,0)[cc]{$\bullet$}}
\put(115,15){\makebox(0,0)[cc]{$\bullet$}}
\put(115,20){\makebox(0,0)[cc]{$\bullet$}}
\put(115,25){\makebox(0,0)[cc]{$\bullet$}}
\put(115,30){\makebox(0,0)[cc]{$\bullet$}}
\dashline{1}(05,05)(05,35)
\dashline{1}(20,35)(05,35)
\dashline{1}(20,35)(20,05)
\dashline{1}(05,05)(20,05)
\dashline{1}(25,05)(25,35)
\dashline{1}(40,35)(25,35)
\dashline{1}(40,35)(40,05)
\dashline{1}(25,05)(40,05)
\dashline{1}(45,05)(45,35)
\dashline{1}(60,35)(45,35)
\dashline{1}(60,35)(60,05)
\dashline{1}(45,05)(60,05)
\dashline{1}(100,05)(100,35)
\dashline{1}(115,35)(100,35)
\dashline{1}(115,35)(115,05)
\dashline{1}(100,05)(115,05)
\dottedline[.]{1}(20.00,10.00)(25.00,10.00)
\dottedline[.]{1}(20.00,15.00)(25.00,15.00)
\dottedline[.]{1}(20.00,20.00)(25.00,20.00)
\dottedline[.]{1}(20.00,25.00)(25.00,25.00)
\dottedline[.]{1}(20.00,30.00)(25.00,30.00)
\dottedline[.]{1}(40.00,10.00)(45.00,10.00)
\dottedline[.]{1}(40.00,15.00)(45.00,15.00)
\dottedline[.]{1}(40.00,20.00)(45.00,20.00)
\dottedline[.]{1}(40.00,25.00)(45.00,25.00)
\dottedline[.]{1}(40.00,30.00)(45.00,30.00)
\drawline(05,10)(20,10)
\drawline(25,10)(40,10)
\drawline(45,10)(60,10)
\drawline(100,10)(115,10)
\drawline(25,15)(40,15)
\drawline(25,20)(40,20)
\drawline(25,25)(40,25)
\drawline(25,30)(40,30)
\linethickness{1pt}
\qbezier(20,15)(15,17.50)(20,20)
\qbezier(20,25)(15,27.50)(20,30)
\qbezier(45,15)(50,17.50)(45,20)
\qbezier(45,25)(50,27.50)(45,30)
\put(70,20){\makebox(0,0)[cc]{$=$}}
\put(85,20){\makebox(0,0)[cc]{$(\delta-1)^2$}}
\put(12.50,10){\makebox(0,0)[cc]{$\blacksquare$}}
\put(32.50,10){\makebox(0,0)[cc]{$\blacksquare$}}
\put(52.50,10){\makebox(0,0)[cc]{$\blacksquare$}}
\put(107.50,10){\makebox(0,0)[cc]{$\blacksquare$}}
\put(32.50,15){\makebox(0,0)[cc]{$\blacksquare$}}
\put(32.50,20){\makebox(0,0)[cc]{$\blacksquare$}}
\put(32.50,25){\makebox(0,0)[cc]{$\blacksquare$}}
\put(32.50,30){\makebox(0,0)[cc]{$\blacksquare$}}
\put(16.50,17.50){\makebox(0,0)[cc]{$\blacksquare$}}
\put(16.50,27.50){\makebox(0,0)[cc]{$\blacksquare$}}
\put(48.50,17.50){\makebox(0,0)[cc]{$\blacksquare$}}
\put(48.50,27.50){\makebox(0,0)[cc]{$\blacksquare$}}
\end{picture}
\caption{Calculation in the proof of Proposition~\ref{prop4}, $n=5$, $k=2$.}\label{fig11} 
\end{figure}

\subsection{The odd summand}\label{s2.3}

To simplify notation in this subsection we set $\omega:=\omega_{\{1\}}$. Define the injective map
$\varphi:\mathtt{B}_{n-1}\to \mathtt{R}_n$ by $\varphi(d)=(u\otimes u^{\star})\otimes d$ for $d\in \mathtt{B}_{n-1}$.

\begin{proposition}\label{prop5}
Assume $\delta'\in\Bbbk^{\times}$, then we have the following:
\begin{enumerate}[$($i$)$]
\item\label{prop5.1} For $d\in\mathtt{R}_n$ we have $\omega d\omega\neq 0$ if and only if
$d\in \varphi(\mathtt{B}_{n-1})$.
\item\label{prop5.2} The set $\{\langle d\rangle:d\in \varphi(\mathtt{B}_{n-1})\}$ 
is a basis of $\omega\mathtt{R}_n\omega$.
\item\label{prop5.3} The $\Bbbk$-linear map $\gamma:\omega\mathfrak{R}_n\omega\to \mathfrak{B}_n(\delta-1)$, 
given by $\langle d\rangle\mapsto \varphi^{-1}(d)$, $d\in \varphi(\mathtt{B}_{n-1})$, is an isomorphism of 
$\Bbbk$-algebras.
\end{enumerate}
\end{proposition}

\begin{proof}
Claim \eqref{prop5.1} follows directly from Proposition~\ref{prop1}\eqref{prop1.2}. Note that for $d\in \mathtt{B}_n$
we have $\langle d\rangle=\omega d\omega$. Hence claim  \eqref{prop5.2} follows from \eqref{prop5.1} and
Proposition~\ref{prop1}\eqref{prop1.1}. Finally, claim  \eqref{prop5.3} follows from \eqref{prop5.2}
and Proposition~\ref{prop1}\eqref{prop1.2}.
\end{proof}

The second direct summand $\mathfrak{R}_n^1$ in the decomposition \eqref{eq652}
can now be described by the following statement:

\begin{proposition}\label{prop6}
Assume $\delta',(\delta-1)\in\Bbbk^{\times}$, then the functor 
\begin{equation}\label{eqf2}
\mathrm{F}_{\omega}:=\mathrm{Hom}_{\mathfrak{R}_n}(\mathfrak{R}_n\omega,{}_{-}): \mathfrak{R}_n\text{-}\mathrm{mod}
\to \omega\mathfrak{R}_n\omega\text{-}\mathrm{mod}
\end{equation}
induces an equivalence $\mathfrak{R}_n^1\text{-}\mathrm{mod}\cong \mathfrak{B}_{n-1}(\delta-1)\text{-}\mathrm{mod}$.
\end{proposition}

\begin{proof}
By \cite[Section~5]{Au}, $\mathrm{F}_{\omega}$ induces an equivalence between the full subcategory $\mathcal{X}$ of 
$\mathfrak{R}_n\text{-}\mathrm{mod}$, consisting of modules $M$ having a presentation $X_1\to X_0\to M\to 0$
with $X_1,X_0\in \mathrm{add}(\mathfrak{R}_n\omega)$, and $\omega\mathfrak{R}_n\omega\text{-}\mathrm{mod}$.
By Proposition~\ref{prop5}, the category $\omega\mathfrak{R}_n\omega\text{-}\mathrm{mod}$ is equivalent to 
$\mathfrak{B}_{n-1}(\delta-1)\text{-}\mathrm{mod}$. So, to complete the proof we have just to show that 
$\mathcal{X}\cong \mathfrak{R}_n^1\text{-}\mathrm{mod}$. For this it is enough to show that the trace 
$\mathfrak{R}_n\omega\mathfrak{R}_n$ of $\mathfrak{R}_n\omega$ in $\mathfrak{R}_n$ coincides with 
$\mathfrak{R}_n^1$. As $\omega\in \mathfrak{R}_n^1$, we have $\mathfrak{R}_n\omega\mathfrak{R}_n\subset
\mathfrak{R}_n^1$. 

As $\mathfrak{R}_n\omega\mathfrak{R}_n$ is an ideal, it is left to show that it contains $1_o$. For this
it is enough to show that $\mathfrak{R}_n\omega\mathfrak{R}_n$ contains $\omega_X$ for each 
$X\subset\underline{n}$ such that $|X|$ is odd (see Proposition~\ref{prop2}\eqref{prop2.2}). 
As $\mathfrak{R}_n\omega\mathfrak{R}_n$ is stable under 
multiplication with the elements of the symmetric group, it is enough to show that 
$\mathfrak{R}_n\omega\mathfrak{R}_n$ contains $\omega_{\{1,2,\dots,2k+1\}}$ for each $k$ such that 
$2k+1\leq n$. Consider the element $x:=(u\otimes u^{\star})\otimes
(u\otimes u\otimes (vu)^{\star})^{\otimes k}\otimes 1^{\otimes (n-2k-1)}$.
Then a direct calculation using Proposition~\ref{prop1} shows that
\begin{displaymath}
{\delta'}^2(\delta-1)^k\omega_{\{1,2,\dots,2k+1\}}= \langle x\rangle\omega\langle x^{\star}\rangle
\end{displaymath}
(see illustration in Figure~\ref{fig12}) and the claim follows inverting $\delta'$ and $\delta-1$.
\end{proof}

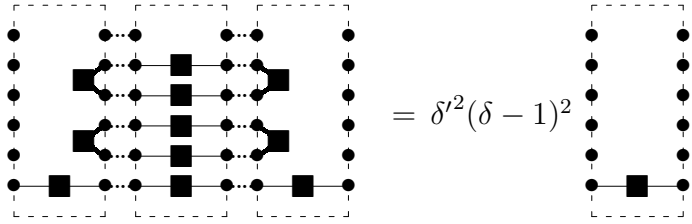
\begin{figure}
\special{em:linewidth 0.4pt} \unitlength 0.80mm
\begin{picture}(125,45)
\put(05,10){\makebox(0,0)[cc]{$\bullet$}}
\put(05,15){\makebox(0,0)[cc]{$\bullet$}}
\put(05,20){\makebox(0,0)[cc]{$\bullet$}}
\put(05,25){\makebox(0,0)[cc]{$\bullet$}}
\put(05,30){\makebox(0,0)[cc]{$\bullet$}}
\put(05,35){\makebox(0,0)[cc]{$\bullet$}}
\put(20,10){\makebox(0,0)[cc]{$\bullet$}}
\put(20,15){\makebox(0,0)[cc]{$\bullet$}}
\put(20,20){\makebox(0,0)[cc]{$\bullet$}}
\put(20,25){\makebox(0,0)[cc]{$\bullet$}}
\put(20,30){\makebox(0,0)[cc]{$\bullet$}}
\put(20,35){\makebox(0,0)[cc]{$\bullet$}}
\put(25,10){\makebox(0,0)[cc]{$\bullet$}}
\put(25,15){\makebox(0,0)[cc]{$\bullet$}}
\put(25,20){\makebox(0,0)[cc]{$\bullet$}}
\put(25,25){\makebox(0,0)[cc]{$\bullet$}}
\put(25,30){\makebox(0,0)[cc]{$\bullet$}}
\put(25,35){\makebox(0,0)[cc]{$\bullet$}}
\put(40,10){\makebox(0,0)[cc]{$\bullet$}}
\put(40,15){\makebox(0,0)[cc]{$\bullet$}}
\put(40,20){\makebox(0,0)[cc]{$\bullet$}}
\put(40,25){\makebox(0,0)[cc]{$\bullet$}}
\put(40,30){\makebox(0,0)[cc]{$\bullet$}}
\put(40,35){\makebox(0,0)[cc]{$\bullet$}}
\put(45,10){\makebox(0,0)[cc]{$\bullet$}}
\put(45,15){\makebox(0,0)[cc]{$\bullet$}}
\put(45,20){\makebox(0,0)[cc]{$\bullet$}}
\put(45,25){\makebox(0,0)[cc]{$\bullet$}}
\put(45,30){\makebox(0,0)[cc]{$\bullet$}}
\put(45,35){\makebox(0,0)[cc]{$\bullet$}}
\put(60,10){\makebox(0,0)[cc]{$\bullet$}}
\put(60,15){\makebox(0,0)[cc]{$\bullet$}}
\put(60,20){\makebox(0,0)[cc]{$\bullet$}}
\put(60,25){\makebox(0,0)[cc]{$\bullet$}}
\put(60,30){\makebox(0,0)[cc]{$\bullet$}}
\put(60,35){\makebox(0,0)[cc]{$\bullet$}}
\put(100,10){\makebox(0,0)[cc]{$\bullet$}}
\put(100,15){\makebox(0,0)[cc]{$\bullet$}}
\put(100,20){\makebox(0,0)[cc]{$\bullet$}}
\put(100,25){\makebox(0,0)[cc]{$\bullet$}}
\put(100,30){\makebox(0,0)[cc]{$\bullet$}}
\put(100,35){\makebox(0,0)[cc]{$\bullet$}}
\put(115,10){\makebox(0,0)[cc]{$\bullet$}}
\put(115,15){\makebox(0,0)[cc]{$\bullet$}}
\put(115,20){\makebox(0,0)[cc]{$\bullet$}}
\put(115,25){\makebox(0,0)[cc]{$\bullet$}}
\put(115,30){\makebox(0,0)[cc]{$\bullet$}}
\put(115,35){\makebox(0,0)[cc]{$\bullet$}}
\dashline{1}(05,05)(05,40)
\dashline{1}(20,40)(05,40)
\dashline{1}(20,40)(20,05)
\dashline{1}(05,05)(20,05)
\dashline{1}(25,05)(25,40)
\dashline{1}(40,40)(25,40)
\dashline{1}(40,40)(40,05)
\dashline{1}(25,05)(40,05)
\dashline{1}(45,05)(45,40)
\dashline{1}(60,40)(45,40)
\dashline{1}(60,40)(60,05)
\dashline{1}(45,05)(60,05)
\dashline{1}(100,05)(100,40)
\dashline{1}(115,40)(100,40)
\dashline{1}(115,40)(115,05)
\dashline{1}(100,05)(115,05)
\dottedline[.]{1}(20.00,10.00)(25.00,10.00)
\dottedline[.]{1}(20.00,15.00)(25.00,15.00)
\dottedline[.]{1}(20.00,20.00)(25.00,20.00)
\dottedline[.]{1}(20.00,25.00)(25.00,25.00)
\dottedline[.]{1}(20.00,30.00)(25.00,30.00)
\dottedline[.]{1}(20.00,35.00)(25.00,35.00)
\dottedline[.]{1}(40.00,10.00)(45.00,10.00)
\dottedline[.]{1}(40.00,15.00)(45.00,15.00)
\dottedline[.]{1}(40.00,20.00)(45.00,20.00)
\dottedline[.]{1}(40.00,25.00)(45.00,25.00)
\dottedline[.]{1}(40.00,30.00)(45.00,30.00)
\dottedline[.]{1}(40.00,35.00)(45.00,35.00)
\drawline(05,10)(20,10)
\drawline(25,10)(40,10)
\drawline(45,10)(60,10)
\drawline(100,10)(115,10)
\drawline(25,15)(40,15)
\drawline(25,20)(40,20)
\drawline(25,25)(40,25)
\drawline(25,30)(40,30)
\linethickness{1pt}
\qbezier(20,15)(15,17.50)(20,20)
\qbezier(20,25)(15,27.50)(20,30)
\qbezier(45,15)(50,17.50)(45,20)
\qbezier(45,25)(50,27.50)(45,30)
\put(82,22){\makebox(0,0)[cc]{$=\,{\delta'}^2(\delta-1)^2$}}
\put(12.50,10){\makebox(0,0)[cc]{$\blacksquare$}}
\put(32.50,10){\makebox(0,0)[cc]{$\blacksquare$}}
\put(52.50,10){\makebox(0,0)[cc]{$\blacksquare$}}
\put(107.50,10){\makebox(0,0)[cc]{$\blacksquare$}}
\put(32.50,15){\makebox(0,0)[cc]{$\blacksquare$}}
\put(32.50,20){\makebox(0,0)[cc]{$\blacksquare$}}
\put(32.50,25){\makebox(0,0)[cc]{$\blacksquare$}}
\put(32.50,30){\makebox(0,0)[cc]{$\blacksquare$}}
\put(16.50,17.50){\makebox(0,0)[cc]{$\blacksquare$}}
\put(16.50,27.50){\makebox(0,0)[cc]{$\blacksquare$}}
\put(48.50,17.50){\makebox(0,0)[cc]{$\blacksquare$}}
\put(48.50,27.50){\makebox(0,0)[cc]{$\blacksquare$}}
\end{picture}
\caption{Calculation in the proof of Proposition~\ref{prop6}, $n=6$, $k=2$.}\label{fig12} 
\end{figure}

Theorem~\ref{thm1} now follows combining Propositions~\ref{prop2}, \ref{prop4} and \ref{prop6}.

\subsection{Application: the Cartan matrix for $\mathfrak{R}_n$ over $\mathbb{C}$}\label{s2.4}

Recall that if $A,B$ are Morita equivalent finite dimensional algebras over a field, then there is a bijection between 
the sets of equivalence classes of simple modules; and this induces an identification of Cartan decomposition 
matrices. For each  Brauer algebra over $\mathbb{C}$ the corresponding Cartan decomposition matrix $C$ is determined 
in  \cite{Martin08-9} (using heavy machinery such as \cite{CoxDevisscherMartin0509}). Thus Theorem~\ref{thm1}
determines the Cartan decomposition matrix, over $\mathbb{C}$, for partial Brauer algebras in case
$\delta'\neq 0$ and $\delta\neq 1$.
Similarly, it is well-known that Brauer algebra over $\mathbb{C}$ are generically semi-simple (see e.g.
\cite{Brown55,CoxDevisscherMartin0509}). Hence Theorem~\ref{thm1} implies:

\begin{corollary}\label{cor7}
The complex algebra $\mathfrak{R}_n(\delta,\delta')$ is generically semisimple.
\end{corollary}

However, knowledge of the Cartan matrix does not lead directly to constructions for simple or indecomposable projective 
modules, or simple characters. To facilitate this for partial Brauer algebras it is convenient to introduce an 
intermediate class of modules (the so-called {\em Specht modules} or {\em Brauer-Specht modules}) 
with a concrete construction, and to tie these also 
to the Brauer algebra case. In \cite{Martin08-9} the Cartan decomposition matrix is determined in the framework of 
a splitting $\pi$-modular system (in the sense of Brauer's modular representation theory, although the prime $\pi$ here 
is a linear monic, not a prime number). That is, the Brauer-Specht module decomposition matrix $D$, which gives the 
simple content of `modular' reductions of the lifts of generic (i.e. in this case $\delta$-indeterminate) simple 
modules, is determined. The Morita equivalence gives a correspondence between Brauer-Specht  modules for the 
partial Brauer algebra and those for the Brauer algebra. 
Thus it only remains to cast the partial Brauer algebras 
in the same framework, construct their Brauer-Specht modules and show
that they 
are preserved by the Morita equivalence.

\section{Specht modules for $\mathfrak{R}_n$}\label{s3}

\subsection{Construction of Specht modules}\label{s3.1}

Specht modules for (integral) partial Brauer algebras, which descend to a complete set of simple modules over complex 
numbers for generic values of parameters, are constructed exactly as for the partition algebra \cite{Martin94}. 
In this section, aiming to reach out to readers in the semigroup community, 
we (show how to) cast this construction in terms of the 
beautiful machinery of semigroup 
representation theory (see e.g. \cite[Section~11]{GM} or \cite{GMS}).

Define the equivalence relation $\sim_L$ on $\mathtt{R}_n$ as follows: $d\sim_L d'$ if and only for any $X\subset\underline{n}'$ the set $X$ is a part in $d$ if and only if it is a part in $d'$. Let $\mathcal{L}$ be an equivalence class for $\sim_L$ (such class is called a {\em left cell}). Let $\mathbf{C}_{\mathcal{L}}$ be the free
$\Bbbk$-submodule of $\mathfrak{R}_n$ with basis $\mathcal{L}$. We turn 
$\mathbf{C}_{\mathcal{L}}$ into an $\mathfrak{R}_n$-module
by setting, for every $d\in\mathtt{R}_n$ and $d'\in\mathcal{L}$,
\begin{displaymath}
d\cdot d'=\begin{cases}
dd',& dd'\in \mathbf{C}_{\mathcal{L}};\\
0,& \text{otherwise}.
\end{cases}
\end{displaymath}

If $\mathcal{L}$ is a left cell, then all elements in $\mathcal{L}$ have the same propagating number, which we
will call the {\em propagating number} of $\mathcal{L}$.

\begin{proposition}\label{prop11}
Let $\Bbbk=\mathbb{Z}[\delta,\delta',\frac{1}{\delta'}]$.
Let $\mathcal{L}$ be a left cell with propagating number $k\in\mathbb{N}_0$.
\begin{enumerate}[$($i$)$]
\item\label{prop11.1} Let $\mathcal{L}'$ be 
another left cells with the same propagating number, then 
$\mathbf{C}_{\mathcal{L}}\cong \mathbf{C}_{\mathcal{L}'}$
as  $\mathfrak{R}_n$-modules.
\item\label{prop11.2} The endomorphism ring of 
$\mathbf{C}_{\mathcal{L}}$ is isomorphic to $\Bbbk[S_k]$.
\item\label{prop11.3} $\mathbf{C}_{\mathcal{L}}$ is free over its endomorphism ring.
\end{enumerate}
\end{proposition}

\begin{proof}
Assume first that $k\in\mathbb{N}$.
The symmetric group $S_n$, given by partitions consisting only of propagating lines, acts on $\mathfrak{R}_n$ via automorphisms by conjugation. Using this action, without loss of generality we may assume that $\mathcal{L}$ contains the
element $\varepsilon_{(a,b,k)}$ for some $a,b\in\mathbb{N}_0$
such that $a+2b+k=n$, defined as follows:
\begin{displaymath}
\varepsilon_{(a,b,k)}=
(u\otimes u^{\star})^{\otimes a}\otimes
((vu)^{\star})^{\otimes b}\otimes 1\otimes
(vu)^{\otimes b}\otimes 1^{\otimes k-1}.
\end{displaymath}
The element $\varepsilon_{(2,2,2)}$ is shown in Figure~\ref{fig21}.
\begin{figure}
 \special{em:linewidth 0.4pt} \unitlength 0.80mm
\begin{picture}(90,55)
\put(05,10){\makebox(0,0)[cc]{$\bullet$}}
\put(05,15){\makebox(0,0)[cc]{$\bullet$}}
\put(05,20){\makebox(0,0)[cc]{$\bullet$}}
\put(05,25){\makebox(0,0)[cc]{$\bullet$}}
\put(05,30){\makebox(0,0)[cc]{$\bullet$}}
\put(05,35){\makebox(0,0)[cc]{$\bullet$}}
\put(05,40){\makebox(0,0)[cc]{$\bullet$}}
\put(05,45){\makebox(0,0)[cc]{$\bullet$}}
\put(25,10){\makebox(0,0)[cc]{$\bullet$}}
\put(25,15){\makebox(0,0)[cc]{$\bullet$}}
\put(25,20){\makebox(0,0)[cc]{$\bullet$}}
\put(25,25){\makebox(0,0)[cc]{$\bullet$}}
\put(25,30){\makebox(0,0)[cc]{$\bullet$}}
\put(25,35){\makebox(0,0)[cc]{$\bullet$}}
\put(25,40){\makebox(0,0)[cc]{$\bullet$}}
\put(25,45){\makebox(0,0)[cc]{$\bullet$}}
\put(35,10){\makebox(0,0)[cc]{$\bullet$}}
\put(35,15){\makebox(0,0)[cc]{$\bullet$}}
\put(35,20){\makebox(0,0)[cc]{$\bullet$}}
\put(35,25){\makebox(0,0)[cc]{$\bullet$}}
\put(35,30){\makebox(0,0)[cc]{$\bullet$}}
\put(35,35){\makebox(0,0)[cc]{$\bullet$}}
\put(35,40){\makebox(0,0)[cc]{$\bullet$}}
\put(35,45){\makebox(0,0)[cc]{$\bullet$}}
\put(55,10){\makebox(0,0)[cc]{$\bullet$}}
\put(55,15){\makebox(0,0)[cc]{$\bullet$}}
\put(55,20){\makebox(0,0)[cc]{$\bullet$}}
\put(55,25){\makebox(0,0)[cc]{$\bullet$}}
\put(55,30){\makebox(0,0)[cc]{$\bullet$}}
\put(55,35){\makebox(0,0)[cc]{$\bullet$}}
\put(55,40){\makebox(0,0)[cc]{$\bullet$}}
\put(55,45){\makebox(0,0)[cc]{$\bullet$}}
\put(65,10){\makebox(0,0)[cc]{$\bullet$}}
\put(65,15){\makebox(0,0)[cc]{$\bullet$}}
\put(65,20){\makebox(0,0)[cc]{$\bullet$}}
\put(65,25){\makebox(0,0)[cc]{$\bullet$}}
\put(65,30){\makebox(0,0)[cc]{$\bullet$}}
\put(65,35){\makebox(0,0)[cc]{$\bullet$}}
\put(65,40){\makebox(0,0)[cc]{$\bullet$}}
\put(65,45){\makebox(0,0)[cc]{$\bullet$}}
\put(85,10){\makebox(0,0)[cc]{$\bullet$}}
\put(85,15){\makebox(0,0)[cc]{$\bullet$}}
\put(85,20){\makebox(0,0)[cc]{$\bullet$}}
\put(85,25){\makebox(0,0)[cc]{$\bullet$}}
\put(85,30){\makebox(0,0)[cc]{$\bullet$}}
\put(85,35){\makebox(0,0)[cc]{$\bullet$}}
\put(85,40){\makebox(0,0)[cc]{$\bullet$}}
\put(85,45){\makebox(0,0)[cc]{$\bullet$}}
\dashline{1}(05,05)(05,50)
\dashline{1}(25,50)(05,50)
\dashline{1}(25,50)(25,05)
\dashline{1}(05,05)(25,05)
\dashline{1}(35,05)(35,50)
\dashline{1}(55,50)(35,50)
\dashline{1}(55,50)(55,05)
\dashline{1}(35,05)(55,05)
\dashline{1}(65,05)(65,50)
\dashline{1}(85,50)(65,50)
\dashline{1}(85,50)(85,05)
\dashline{1}(65,05)(85,05)
\drawline(05,10)(25,10)
\drawline(05,35)(25,15)
\drawline(65,10)(85,10)
\drawline(65,25)(85,15)
\drawline(35,10)(55,10)
\drawline(35,15)(55,15)
\drawline(35,20)(55,20)
\drawline(35,35)(55,25)
\linethickness{1pt}
\qbezier(05,15)(10,17.50)(05,20)
\qbezier(05,25)(10,27.50)(05,30)
\qbezier(25,20)(20,22.50)(25,25)
\qbezier(25,35)(20,32.50)(25,30)
\qbezier(35,25)(40,27.50)(35,30)
\qbezier(65,15)(70,17.50)(65,20)
\qbezier(85,20)(80,22.50)(85,25)
\end{picture}
\caption{The elements $\varepsilon_{(2,2,2)}$,
$\varphi_{(2,4)}$ and $\varepsilon_{(4,1,2)}$}\label{fig21} 
\end{figure}
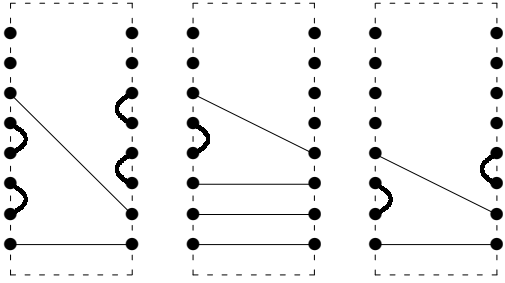
It is easy to check that the element 
$\varepsilon'_{(a,b,k)}:=\frac{1}{{\delta'}^a}\varepsilon_{(a,b,k)}$ 
is an idempotent
and that it generates $\mathbf{C}_{\mathcal{L}}$.

Let $\mathcal{H}$ denote the set of diagrams $d\in\mathcal{L}$
satisfying $d^\star\sim_L \varepsilon_{(a,b,k)}^\star$.
From the previous paragraph it follows that the endomorphism ring of $\mathbf{C}_{\mathcal{L}}$ is isomorphic to 
$\varepsilon'_{(a,b,k)}\mathbf{C}_{\mathcal{L}}$. The latter
is a free $\Bbbk$-module with basis $\{\frac{1}{{\delta'}^a}d:d\in \mathcal{H}\}$. From \cite[Theorem~1]{Mazorchuk95}
it follows that the elements of this basis form a group
isomorphic to $S_k$. This proves claim \eqref{prop11.2}.
Claim \eqref{prop11.3} follows from claim \eqref{prop11.2} by a direct calculation if we, for example, take, as a basis
of $\mathbf{C}_{\mathcal{L}}$ over $\Bbbk[S_n]$, the set
of all diagrams in $\mathcal{L}$ in which propagating lines
can be drawn such that they do not cross each other.

To prove claim \eqref{prop11.1}, without loss of generality
we may assume that $\mathbf{C}_{\mathcal{L}'}$ is generated
by $\varepsilon'_{(c,d,k)}$ and $c>a$. Then for the
element
\begin{displaymath}
\varphi_{(a,c)}:= (u\otimes u^{\star})^{\otimes a}
\otimes (u^{\star})^{\otimes c-a}\otimes 1 \otimes
(vu)^{\otimes c-a}\otimes 1^{\otimes k+2(c-a)-1}
\end{displaymath}
(see Figure~\ref{fig21})
we have $\varepsilon'_{(a,b,k)}\varphi_{(a,c)}=
\frac{1}{{\delta'}^a}\varepsilon'_{(c,d,k)}$, which implies
that the right multiplication with $\varphi_{(a,c)}$
induces the desired isomorphism between
$\mathbf{C}_{\mathcal{L}}$ and $\mathbf{C}_{\mathcal{L}'}$.

If $k=0$, we let $\varepsilon_{(a,b,0)}$ be the element in $\mathcal{L}$ in which
$\underline{n}$ is split into singletons and set
$\varepsilon'_{(a,b,0)}=\frac{1}{{\delta'}^{a+b}}\varepsilon_{(a,b,0)}$ and
$\varphi_{(a,c)}=\varepsilon_{(c,d,0)}$. Using these elements
the rest of the proof is similar to the above.
\end{proof}

For any partition $\lambda\vdash k$ let $\mathcal{S}_{\mathbb{Z}}(\lambda)$ denote the Specht module
for $\mathbb{Z}[S_k]$ corresponding to $\lambda$ and let
$\mathcal{S}_{\Bbbk}(\lambda)$ denote the corresponding Specht module
for $\Bbbk[S_k]$ (with  scalars induced). We define the 
{\em integral Specht module} $\Delta_{\Bbbk}(\lambda)$
for the algebra $\mathfrak{R}_n$ as follows:
\begin{displaymath}
\Delta_{\Bbbk}(\lambda):=
\mathbf{C}_{\mathcal{L}}\otimes_{\Bbbk[S_k]}
\mathcal{S}_{\Bbbk}(\lambda),
\end{displaymath}
where $\mathcal{L}$ is some fixed left cell with propagating
number $k$. Here the right action of $\Bbbk[S_k]$ on
$\mathbf{C}_{\mathcal{L}}$ is given by 
Proposition~\ref{prop11}\eqref{prop11.2}.
By Proposition~\ref{prop11}\eqref{prop11.3}, this definition does not depend on the choice of $\mathcal{L}$, up to isomorphism.

The intersection of a left cell $\mathcal{L}$ with
$\mathtt{B}_n$ gives a {\em left cell} $\mathcal{L}^{\mathfrak{B}}$ for $\mathfrak{B}_n$
(note that $\mathcal{L}^{\mathfrak{B}}\neq\varnothing$
implies that the propagating number of $\mathcal{L}$
has the same parity as $n$). Similarly to the above
we define modules $\mathbf{C}_{\mathcal{L}^{\mathfrak{B}}}$
and the {\em integral Specht module} 
$\Delta_{\Bbbk}^{\mathfrak{B}}(n,\lambda)$
for the algebra $\mathfrak{B}_n$ as follows:
\begin{displaymath}
\Delta_{\Bbbk}^{\mathfrak{B}}(n,\lambda):=
\mathbf{C}_{\mathcal{L}^{\mathfrak{B}}}\otimes_{\Bbbk[S_k]}
\mathcal{S}_{\Bbbk}(\lambda),
\end{displaymath}
using the obvious analogue of Proposition~\ref{prop11}
for the algebra $\mathfrak{B}_n$. If $n$ is clear from the context, we will sometimes write
simply $\Delta_{\Bbbk}^{\mathfrak{B}}(\lambda)$ for 
$\Delta_{\Bbbk}^{\mathfrak{B}}(n,\lambda)$.

\subsection{Specht modules and Morita equivalence}\label{s3.2}

Our first important observation about Specht modules is:

\begin{proposition}\label{prop12}
Assume $\delta'\in\Bbbk^{\times}$. Then both the functor $\mathrm{F}$ from Proposition~\ref{prop4} 
and the functor $\mathrm{F}_{\omega}$ from Proposition~\ref{prop6}
map Specht modules to Specht modules.
\end{proposition}

\begin{proof}
We prove the claim for the functor $\mathrm{F}$ (for
$\mathrm{F}_{\omega}$ one can use similar arguments). Choose
$\mathcal{L}$ such that $\mathcal{L}^{\mathfrak{B}}\neq
\varnothing$. Clearly, it is enough to show that 
$\mathrm{F}$ maps $\mathbf{C}_{\mathcal{L}}$ to
$\mathbf{C}_{\mathcal{L}^{\mathfrak{B}}}$. From the
definition of $\mathrm{F}$ it follows that for any module $M$
we have $\mathrm{F}\, M=\hat{u}_n M$ (as a $\Bbbk$-vector space). 
Similarly to Proposition~\ref{prop1}\eqref{prop1.2} one shows that 
for $d\in \mathcal{L}$ we have $\hat{u}_nd\neq 0$ if and only if 
$d\in \mathcal{L}^{\mathfrak{B}}$. This implies that 
$\mathrm{F}\,\mathbf{C}_{\mathcal{L}}\cong
\mathbf{C}_{\mathcal{L}^{\mathfrak{B}}}$ as
a $\Bbbk$-vector space. It is straightforward to check that
this isomorphism intertwines the actions of 
$\mathfrak{R}_n$ and $\mathfrak{B}_n$. The claim follows.
\end{proof}

\begin{remark}\label{rem1305}
{\rm  
Simple characters for the complex algebra $\mathfrak{R}_n(\delta,\delta')$ 
in the case $\delta',\delta-1\in\mathbb{C}^{\times}$ can be obtained combining
Theorem~\ref{thm1}, Proposition~\ref{prop12} and \cite{Martin08-9}. Indeed,
Theorem~\ref{thm1} and \cite{Martin08-9} determine the decomposition matrix for
Specht modules while Proposition~\ref{prop12} and the combinatorial construction 
of Specht modules for $\mathfrak{R}_n(\delta,\delta')$ give us dimensions of
Specht modules.
}
\end{remark}

\begin{corollary}\label{cor14}
Let $\Bbbk$ be an algebraically closed field. Assume that the characteristics
of $\Bbbk$ does not divide  $n!$ and that 
$\delta'  \neq 0$. 
Then the $\Bbbk$-algebra $\mathfrak{R}_n$ 
is quasi-hereditary with Specht modules
$\Delta_{\Bbbk}(\lambda)$ being standard.
\end{corollary}

\begin{proof}
As $\mathtt{R}_n$ is a disjoint union of cells, it
follows from Proposition~\ref{prop11}\eqref{prop11.1}
and 
the definition of $\Delta_{\Bbbk}(\lambda)$  
that the regular representation of 
$\mathfrak{R}_n$ has a filtration with quotients 
isomorphic to Specht modules. 

If $\delta'  \neq 0$
the endomorphism ring of
$\Delta_{\Bbbk}(\lambda)$ is isomorphic to the endomorphism
ring of $\mathcal{S}_{\Bbbk}(\lambda)$ and hence
coincides with $\Bbbk$.

Finally, for any linear order on the set of 
all partitions $\lambda\vdash k$, $k\leq n$, satisfying
$\lambda'<\lambda$ provided that  $\lambda\vdash k$,
$\lambda'\vdash k'$ and $k<k'$, using the arguments
from the proof of \cite[Theorem~7]{GMS}, one shows
that all composition subquotients of 
$\Delta_{\Bbbk}(\lambda)$ are indexed by 
$\lambda'$ such that $\lambda'<\lambda$. The claim follows.
\end{proof}

\begin{corollary}\label{cor15}
Specht modules for the generic $\mathbb{C}$-algebra 
$\mathfrak{R}_n$ are semisimple.
\end{corollary}

\begin{proof}
It is known that Specht modules for the $\mathbb{C}$-algebra  $\mathfrak{B}_n$ are generically semisimple (see 
e.g. \cite{Brown55,CoxDevisscherMartin0509}). 
Now the claim follows from Proposition~\ref{prop12} and 
Theorem~\ref{thm1}. 
Alternatively one can proceed by constructing  
a contravariant form on the Specht module,
and show that it is generically non-degenerate.
\end{proof}

\subsection{The branching rule for Specht modules}\label{s3.3}

For a left cell $\mathcal{L}$ denote by $\mathcal{L}_{nc}$ the subset of $\mathcal{L}$ consisting of all diagrams
which can be drawn such that the propagating lines do not cross each other (inside the diagram). For
$\lambda\vdash k$ choose a basis  $\mathtt{b}_{\lambda}=\{b_1,b_2,\dots, b_{k_{\lambda}}\}$ of 
$\mathcal{S}_{\Bbbk}(\lambda)$. From the proof of Proposition~\ref{prop11} it then follows that the set 
\begin{displaymath}
\mathtt{B}_{\mathcal{L},\mathtt{b}_{\lambda}}:=
\{d\otimes b\, \colon \, d\in\mathcal{L}_{nc}, b\in \mathtt{b}_{\lambda} \} 
\end{displaymath}
forms a basis of $\Delta_{\Bbbk}(n,\lambda)$.

The map $d\mapsto d\otimes 1$, $d\in \mathtt{R}_n$, extends uniquely to an injective algebra homomorphism 
$\mathfrak{i}=\mathfrak{i}_{n-1}:\mathfrak{R}_{n-1}\to \mathfrak{R}_n$. This induces the 
restriction map $\mathfrak{i}^*:\mathfrak{R}_n\text{-}\mathrm{mod}\to \mathfrak{R}_{n-1}\text{-}\mathrm{mod}$.
The aim of this subsection is to establish the branching rule for Specht modules with respect to this restriction.
Consider the partition
\begin{displaymath}
\mathtt{B}_{\mathcal{L},\mathtt{b}_{\lambda}}=
\mathtt{B}_{\mathcal{L},\mathtt{b}_{\lambda}}^{(1)}\cup
\mathtt{B}_{\mathcal{L},\mathtt{b}_{\lambda}}^{(2)}\cup
\mathtt{B}_{\mathcal{L},\mathtt{b}_{\lambda}}^{(3)},
\end{displaymath}
where 
\begin{itemize}
\item $\mathtt{B}_{\mathcal{L},\mathtt{b}_{\lambda}}^{(1)}$ consists of all $d\otimes b$ such that 
$\{n\}$ is a singleton for $d$;
\item $\mathtt{B}_{\mathcal{L},\mathtt{b}_{\lambda}}^{(2)}$ consists of all 
$d\otimes b$ such that $n\in\underline{n}$ belongs to a propagating line in $d$;
\item $\mathtt{B}_{\mathcal{L},\mathtt{b}_{\lambda}}^{(3)}$ consists of all $d\otimes b$ such that 
$n\in\underline{n}$ belongs  to a non-propagating pair part in $d$.
\end{itemize}
For $i=1,2,3$ denote by $\Delta_{\Bbbk}^{(i)}(n,\lambda)$
the $\Bbbk$-linear span of $\mathtt{B}_{\mathcal{L},\mathtt{b}_{\lambda}}^{(i)}$. We will analyze each
$\Delta_{\Bbbk}^{(i)}(n,\lambda)$ separately, starting with $\Delta_{\Bbbk}^{(1)}(n,\lambda)$.

\begin{lemma}\label{lem21}
The $\Bbbk$-module $\Delta_{\Bbbk}^{(1)}(n,\lambda)$ is invariant under the action of 
$\mathfrak{i}(\mathfrak{R}_{n-1})$ and is isomorphic to $\Delta_{\Bbbk}(n-1,\lambda)$
as an $\mathfrak{R}_{n-1}$-module.
\end{lemma}

\begin{proof}
The invariance is straightforward. To prove the isomorphism, we first note that the propagating number $k$
of any $d\in \mathcal{L}_{nc}$ is strictly smaller than $n$ since $\{n\}$ is a singleton. Without loss of generality
we may assume that any $d\in \mathcal{L}_{nc}$ contains a singleton in $\underline{n}'$. Fix some such
singleton $\{s'\}$. Then, deleting $\{s'\}$ and $\{n\}$ from $d$ defines a bijective map from 
$\mathtt{B}_{\mathcal{L},\mathtt{b}_{\lambda}}^{(1)}$ to a basis in $\Delta_{\Bbbk}(n-1,\lambda)$,
moreover, the linearization of this map is compatible with the action of $\mathfrak{R}_{n-1}$ by construction. 
The claim follows.
\end{proof}

For two partitions $\mu\vdash m$ and $\nu\vdash m-1$ we write $\nu\to \mu$ provided that the Young diagram of
$\nu$ is obtained from the Young diagram for $\mu$ by removing a removable node (or, equivalently, the Young 
diagram of $\mu$ is obtained from the Young diagram for $\nu$ by inserting an insertable node), see e.g. 
\cite[2.8]{Sa}.

\begin{lemma}\label{lem22}
The $\Bbbk$-module $\Delta_{\Bbbk}^{(2)}(n,\lambda)$ is invariant under the action of 
$\mathfrak{i}(\mathfrak{R}_{n-1})$ and is isomorphic to 
\begin{displaymath}
\bigoplus_{\mu\to\lambda}\Delta_{\Bbbk}(n-1,\mu)
\end{displaymath}
as an $\mathfrak{R}_{n-1}$-module.
\end{lemma}

\begin{proof}
The invariance is again straightforward. As $n$ belongs to a propagating line in every $d\in \mathcal{L}_{nc}$
and this line is not moved by any element from $\mathfrak{i}(\mathfrak{R}_{n-1})$, the restriction 
$\mathfrak{i}^*$ induces, for the right action of the symmetric group $S_k$ on $\Delta_{\Bbbk}(n,\lambda)$,
the restriction to a subgroup isomorphic to $S_{k-1}$. It is well know, see for example \cite[2.8]{Sa},  
that the branching rule for the effect of the latter restriction on Specht modules is as follows:
\begin{displaymath}
\mathcal{S}_{\Bbbk}(\lambda){\cong} 
\bigoplus_{\mu\to\lambda} \mathcal{S}_{\Bbbk}(\mu)
\end{displaymath}
(this is an isomorphism of ${S_{k-1}}$-modules).
Assume that the basis $\mathtt{b}_{\lambda}$ is chosen such that it is compatible with this branching. 
Then, deleting the propagating line containing $n$ from
$d\in \mathcal{L}_{nc}$ defines a bijective map from $\mathtt{B}_{\mathcal{L},\mathtt{b}_{\lambda}}^{(2)}$ to 
the disjoint union of bases in the $\mathfrak{R}_{n-1}$-modules
$\Delta_{\Bbbk}(n-1,\mu)$, where $\mu\to\lambda$. The linearization of
this map is compatible with the action of $\mathfrak{R}_{n-1}$ and hence gives the desired isomorphism.
\end{proof}

Lemmata~\ref{lem21} and \ref{lem22} allow us to consider the $\mathfrak{i}(\mathfrak{R}_{n-1})$-submodule
\begin{displaymath}
X:=\Delta_{\Bbbk}^{(1)}(n,\lambda)\oplus \Delta_{\Bbbk}^{(2)}(n,\lambda)\subset \Delta_{\Bbbk}(n,\lambda). 
\end{displaymath}
The set $\mathtt{B}_{\mathcal{L},\mathtt{b}_{\lambda}}^{(3)}$ naturally becomes a basis of the quotient
$\Delta_{\Bbbk}(n,\lambda)/X$. Thus, the latter quotient can be identified with $\Delta_{\Bbbk}^{(3)}(n,\lambda)$
(as $\Bbbk$-vector space).

\begin{lemma}\label{lem23}
The $\mathfrak{R}_{n-1}$-module $\mathfrak{i}^*(\Delta_{\Bbbk}(n,\lambda)/X)$  is isomorphic to 
\begin{displaymath}
\bigoplus_{\lambda\to\mu}\Delta_{\Bbbk}(n-1,\mu).
\end{displaymath}
\end{lemma}

\begin{proof}
First we note that $\mathtt{B}_{\mathcal{L},\mathtt{b}_{\lambda}}^{(3)}$ is non-empty only in the case $k\leq n-2$.
Hence, without loss of generality, we may assume that each $d\in \mathcal{L}_{nc}$ contains the non-propagating
pair part $\{(n-1)',n'\}$. For $d\in \mathcal{L}_{nc}$ containing the non-propagating pair part $\{s,n\}$
denote by $\gamma(d)$ the diagram in $\mathtt{B}_{n-1}$ obtained by substituting the parts $\{s,n\}$ and
$\{(n-1)',n'\}$ by the propagating line $\{s,(n-1)'\}$ (this is illustrated by Figure~\ref{fig27}). 
Note that $\gamma(d)$ has propagating number $k+1$, furthermore, the newly created propagating line
$\{s,(n-1)'\}$ may cross other propagating lines.

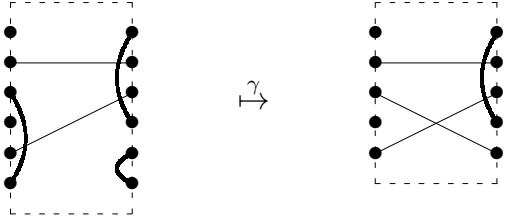
\begin{figure}
\special{em:linewidth 0.4pt} \unitlength 0.80mm
\begin{picture}(90,55)
\put(05,10){\makebox(0,0)[cc]{$\bullet$}}
\put(05,15){\makebox(0,0)[cc]{$\bullet$}}
\put(05,20){\makebox(0,0)[cc]{$\bullet$}}
\put(05,25){\makebox(0,0)[cc]{$\bullet$}}
\put(05,30){\makebox(0,0)[cc]{$\bullet$}}
\put(05,35){\makebox(0,0)[cc]{$\bullet$}}
\put(25,10){\makebox(0,0)[cc]{$\bullet$}}
\put(25,15){\makebox(0,0)[cc]{$\bullet$}}
\put(25,20){\makebox(0,0)[cc]{$\bullet$}}
\put(25,25){\makebox(0,0)[cc]{$\bullet$}}
\put(25,30){\makebox(0,0)[cc]{$\bullet$}}
\put(25,35){\makebox(0,0)[cc]{$\bullet$}}
\put(65,15){\makebox(0,0)[cc]{$\bullet$}}
\put(65,20){\makebox(0,0)[cc]{$\bullet$}}
\put(65,25){\makebox(0,0)[cc]{$\bullet$}}
\put(65,30){\makebox(0,0)[cc]{$\bullet$}}
\put(65,35){\makebox(0,0)[cc]{$\bullet$}}
\put(85,15){\makebox(0,0)[cc]{$\bullet$}}
\put(85,20){\makebox(0,0)[cc]{$\bullet$}}
\put(85,25){\makebox(0,0)[cc]{$\bullet$}}
\put(85,30){\makebox(0,0)[cc]{$\bullet$}}
\put(85,35){\makebox(0,0)[cc]{$\bullet$}}
\dashline{1}(05,05)(05,40)
\dashline{1}(25,40)(05,40)
\dashline{1}(25,40)(25,05)
\dashline{1}(05,05)(25,05)
\dashline{1}(65,10)(65,40)
\dashline{1}(85,40)(65,40)
\dashline{1}(85,40)(85,10)
\dashline{1}(65,10)(85,10)
\drawline(05,15)(25,25)
\drawline(05,30)(25,30)
\drawline(65,15)(85,25)
\drawline(65,30)(85,30)
\drawline(65,25)(85,15)
\linethickness{1pt}
\qbezier(05,10)(10,17.50)(05,25)
\qbezier(25,10)(20,12.50)(25,15)
\qbezier(25,20)(20,27.50)(25,35)
\qbezier(85,20)(80,27.50)(85,35)
\put(45,25){\makebox(0,0)[cc]{$\overset{\gamma}{\mapsto}$}}
\end{picture}
\caption{The map $\gamma$}\label{fig27} 
\end{figure}

As $\gamma(d)$ now has propagating number $k+1$, we have the natural right action of $S_{k+1}$ on the
corresponding cell and Specht modules, with the action of the subgroup $S_{k}$ induced from $\mathbf{C}_{\mathcal{L}}$.
There is a unique element $w\in S_{k+1}$ such that $\gamma(d)w$ can be written such that its propagating
lines do not cross. Now, mapping $d\otimes b$ to $\gamma(d)w\otimes (w^{-1}S_{k}\otimes b)$ defines a bijective map 
from $\mathtt{B}_{\mathcal{L},\mathtt{b}_{\lambda}}^{(3)}$ to the disjoint union of bases in the 
$\mathfrak{R}_{n-1}$-modules $\Delta_{\Bbbk}(n-1,\mu)$, where $\lambda\to\mu$. 
The linearization of this map is compatible with the action of 
$\mathfrak{R}_{n-1}$ by construction and hence gives the desired isomorphism.
\end{proof}

From Lemmata~\ref{lem21}--\ref{lem23} we obtain, as a corollary, the following branching rule 
for restriction of Specht modules.

\begin{theorem}\label{thm24}
For and $k\leq n$ and $\lambda\vdash k$ there is a short exact sequence of $\mathfrak{R}_{n-1}$-modules as follows:
\begin{equation}\label{eq26}
0\to
\Delta_{\Bbbk}(n-1,\lambda)\oplus
\bigoplus_{\mu\to\lambda} \Delta_{\Bbbk}(n-1,\mu)\to
\mathfrak{i}^*(\Delta_{\Bbbk}(n,\lambda))\to
\bigoplus_{\lambda\to\mu} \Delta_{\Bbbk}(n-1,\mu)\to 0.
\end{equation}
\end{theorem}

Note that sequence \eqref{eq26} splits whenever $\mathfrak{R}_{n-1}$ is semisimple. 

Let $\mathcal{Y}$ be the directed Young graph, that is the graph with vertices $\lambda$, where 
$\lambda\vdash k$, $k\in\mathbb{N}_0$; and directed edges $\mu\to\lambda$ as defined above (i.e. $\lambda$ is
obtained from $\mu$ by inserting an insertable node). Denote by $\mathbf{M}$ the incidence matrix of $\mathcal{Y}$.

\begin{corollary}\label{cor25}
For any $\lambda\vdash k$ with $k<n$, the $\Bbbk$-dimension of $\Delta_{\Bbbk}(n,\lambda)$ is given by the
$(\emptyset,\lambda)$-entry of the matrix $(\mathbf{M}+\mathbf{M}^t+\mathbbm{1})^n$, where $\mathbbm{1}$ 
denotes the identity matrix.
\end{corollary}

\begin{proof}
Let $\mathcal{Y}'$ be the directed graph obtained from $\mathcal{Y}$ by first adding reverses of all arrows
and then adding all loops. Then $\mathbf{M}':=\mathbf{M}+\mathbf{M}^t+\mathbbm{1}$ is the incidence matrix 
of $\mathcal{Y}'$. By Theorem~\ref{thm24}, the $\Bbbk$-dimension of $\Delta_{\Bbbk}(n,\lambda)$ is 
given by the number of paths of length $n$ in $\mathcal{Y}'$ from $\lambda$ to $\emptyset$. This is exactly the
$(\emptyset,\lambda)$-entry in $(\mathbf{M}')^n$.
\end{proof}

\section{Schur-Weyl duality}\label{s5}

Let $k\in\mathbb{N}$. Consider the complex vector space $V:=\mathbb{C}^k$ endowed with a non-degenerate
symmetric $\mathbb{C}$-bilinear form $(\cdot,\cdot)$ and let $\mathbf{e}:=(e_1,e_2,\dots,e_k)$ be an
orthonormal basis in $V$. We also endow $\mathbb{C}$ with a non-degenerate symmetric $\mathbb{C}$-bilinear 
form $(\cdot,\cdot)$ such that $(1,1)=1$ and set $e_0:=1$ (this is an orthonormal basis in $\mathbb{C}$).
Let $\mathcal{O}_k$ denote the group of orthogonal transformations of $V$, which makes $V$ an
$\mathcal{O}_k$-module in the natural way. We further consider $\mathbb{C}$ as the trivial $\mathcal{O}_k$-module.
This allows us to consider the $\mathcal{O}_k$-module $\overline{V}:=\mathbb{C}\oplus V$. Note that 
$\mathcal{O}_k$ acts by orthogonal transformations of $\overline{V}$ and $\overline{\mathbf{e}}:=
(e_0,e_1,e_2,\dots,e_k)$ is an orthonormal basis in $\overline{V}$. Now for any $m\in\mathbb{N}$ we can
make
\begin{displaymath}
\overline{V}^{\otimes m}:=\underbrace{\overline{V}\otimes\overline{V}\otimes\cdots
\otimes\overline{V}}_{m\text{ times }}
\end{displaymath}
into an $\mathcal{O}_k$-module using the diagonal action.

The symmetric group $S_m$ acts by automorphisms on the $\mathcal{O}_k$-module $\overline{V}^{\otimes m}$
permuting factors of the tensor product. Consider the endomorphism $\varepsilon$ of 
$\overline{V}:=\mathbb{C}\oplus V$ given by the matrix
\begin{displaymath}
\left(\begin{array}{cc}\mathrm{id}_{\mathbb{C}}&0\\0&0\end{array}\right)
\end{displaymath}
and for $i=1,2,\dots,m$ let $\varepsilon_i$ denote the endomorphism 
\begin{displaymath}
\mathrm{id}\otimes \mathrm{id}\otimes\cdots\otimes\mathrm{id}\otimes\varepsilon
\otimes\mathrm{id}\otimes\mathrm{id}\otimes\cdots\otimes\mathrm{id} 
\end{displaymath}
of $\overline{V}^{\otimes m}$ for which $\varepsilon$ is in position $i$ from the left
(this is taken from \cite{So}).
For $i=1,2,\dots,m-1$ consider also the endomorphism $\zeta_i$ of $\overline{V}^{\otimes m}$  
sending $e_{j_1}\otimes e_{j_2}\otimes\cdots\otimes e_{j_m}$, $j_1,j_2,\dots,j_m\in\{0,1,2,\dots,k\}$, to
\begin{displaymath}
\delta_{j_i,j_{i+1}}\sum_{s=0}^k e_{j_1}\otimes e_{j_2}\otimes\cdots\otimes e_{j_{i-1}}
\otimes e_s\otimes e_s\otimes e_{j_{i+2}}\otimes e_{j_{i+3}}\otimes  \cdots\otimes e_{j_m}
\end{displaymath}
(this generalizes \cite{Brauer37}). Denote by $R_m$ the subalgebra of 
$\mathrm{End}_{\mathbb{C}}(\overline{V}^{\otimes m})$ 
generated by the endomorphisms of the $\mathcal{O}_k$-module $\overline{V}^{\otimes m}$ defined above.

\begin{proposition}\label{prop111}
The algebra $R_m$ coincides with $\mathrm{End}_{\mathcal{O}_k}(\overline{V}^{\otimes m})$.
\end{proposition}

To prove this statement we will need some facts from the classical representation theory of $\mathcal{O}_k$.
We use \cite{GW} as a general reference. Simple $\mathcal{O}_k$-modules appearing in $\overline{V}^{\otimes l}$,
$l\in\mathbb{N}_0$, are indexes by integer partitions $\lambda=(\lambda_1,\lambda_2,\dots,\lambda_k)$,
$\lambda_1\geq \lambda_2\geq \dots\geq\lambda_k\geq 0$, such that $\lambda_1+\lambda_2\leq k$. Let $\Lambda$ denote the
(finite!) set of all such partitions. For $\lambda\in \Lambda$ we denote by $L_{\lambda}$ the corresponding 
simple $\mathcal{O}_k$-module (in particular, $L_{(0)}\cong\mathbb{C}$ and $L_{(1)}\cong V$). Further,
by \cite[Corollary~2.5.3]{KT} we have
\begin{equation}\label{eq112}
V\otimes L_{\lambda}\cong \bigoplus_{\mu\in \Lambda,\,\mu\to\lambda} L_{\mu}\,\,\,\oplus  \,\,\,
\bigoplus_{\mu\in \Lambda,\,\lambda\to\mu} L_{\mu}. 
\end{equation}

\begin{proof}[Proof of Proposition~\ref{prop111}.]
We obviously have $R_m\subset \mathrm{End}_{\mathcal{O}_k}(\overline{V}^{\otimes m})$ and hence we need
only to check the reverse inclusion. 

For $X\subset \{1,2,\dots,m\}$ let $V^{\otimes X}$ denote the tensor product $V_1\otimes V_2\otimes\cdots\otimes
V_m$, where $V_i=V$ if $i\in X$ and $V_i=\mathbb{C}$ otherwise. Using biadditivity of both the tensor product and 
$\mathrm{Hom}$ functor, we can write
\begin{displaymath}
\mathrm{End}_{\mathcal{O}_k}(\overline{V}^{\otimes m})\cong
\bigoplus_{X,Y\subset \{1,2,\dots,m\}}
\mathrm{Hom}_{\mathcal{O}_k}(V^{\otimes X},V^{\otimes Y}).
\end{displaymath}
Using the action of $S_m$ together with endomorphisms $\varepsilon_i$, we only need to check that for
any $i,j\in\{1,2,\dots,m\}$ the component
\begin{displaymath}
\mathrm{Hom}_{\mathcal{O}_k}(V^{\otimes \{1,2,\dots,i\}},V^{\otimes \{1,2,\dots,j\}}) 
\end{displaymath}
belongs to $R_m$. First we observe that \eqref{eq112} implies that for this component to be nonzero,
the elements $i$ and $j$ should have the same parity. In the case $i=j$ the claim follows directly from \cite{Brauer37}. 
By \eqref{eq112}, the $\mathcal{O}_k$-module $V\otimes V$ has a unique
direct summand isomorphic to $\mathbb{C}\otimes\mathbb{C}$. This yields that
\begin{displaymath}
\begin{array}{cc}
\mathrm{Hom}_{\mathcal{O}_k}(V^{\otimes \{1,2,\dots,i\}},V^{\otimes \{1,2,\dots,j\}})\tto  
\mathrm{Hom}_{\mathcal{O}_k}(V^{\otimes \{1,2,\dots,j\}},V^{\otimes \{1,2,\dots,j\}}), & j>i;\\ 
\mathrm{Hom}_{\mathcal{O}_k}(V^{\otimes \{1,2,\dots,i\}},V^{\otimes \{1,2,\dots,j\}})\subset  
\mathrm{Hom}_{\mathcal{O}_k}(V^{\otimes \{1,2,\dots,i\}},V^{\otimes \{1,2,\dots,i\}}), & i>j. 
\end{array}
\end{displaymath}
As noted above, the left hand side belongs to $R_m$ by \cite{Brauer37}. The claim follows.
\end{proof}

For $i=1,2,\dots,m-1$ denote by $p_i$ the element
\begin{displaymath}
\underbrace{1\otimes 1\otimes\cdots\otimes 1}_{i-1\,\text{ factor}}
\otimes vu\otimes (vu)^{\star}\otimes 
\underbrace{1\otimes 1\otimes\cdots\otimes 1}_{m-i-1\,\text{ factor}}\in \mathtt{R}_m.
\end{displaymath}
Combining \cite{Brauer37} with \cite{KudryavtsevaMazorchuk06}, we have:

\begin{proposition}[Presentation for $\mathfrak{R}_m(\delta,\delta')$]\label{prop115}
A presentations of the algebra $\mathfrak{R}_m(\delta,\delta')$ is given by Coxeter generators 
$s_i=(i,i+1)$, $i=1,2,\dots,m-1$, of $S_m$; together with elements $\varepsilon_i$, $i=1,2,\dots,m$;
and $p_i$, $i=1,2,\dots,m-1$; satisfying the following defining relations (for all indices for
which the corresponding relations make sense):
\begin{gather*}
s_i^2=1;\quad s_is_j=s_js_i, |i-j|\neq 1;\quad s_is_{i+1}s_i=s_{i+1}s_is_{i+1};\\
p_i^2=\delta p_i;\quad p_ip_j=p_jp_i, |i-j|\neq 1;\quad p_ip_{j}p_i=p_i, |i-j|=1;\\
p_is_i=s_ip_i=p_i;\quad p_is_j=s_jp_i, |i-j|\neq 1;\quad s_ip_jp_i=s_jp_i, p_ip_js_i=p_is_j,  |i-j|=1;\\
\varepsilon_i^2=\delta' \varepsilon_i;\quad \varepsilon_i\varepsilon_j=\varepsilon_j\varepsilon_i, i\neq j;\quad
\varepsilon_i s_i\varepsilon_i=\varepsilon_i\varepsilon_{i+1};\quad
s_i\varepsilon_i=\varepsilon_{i+1}s_i; \quad s_i\varepsilon_j=\varepsilon_{j}s_i, j\neq i,i+1;\\
p_i\varepsilon_j=\varepsilon_j p_i, j\neq i,i+1;\quad
p_i\varepsilon_i p_i=\delta' p_i;\quad \varepsilon_i p_i\varepsilon_i=\varepsilon_i\varepsilon_{i+1};\\
p_i\varepsilon_i=p_i\varepsilon_{i+1};\quad  p_i\varepsilon_i\varepsilon_{i+1}=\delta' p_i\varepsilon_i;\quad
\varepsilon_ip_i=\varepsilon_{i+1}p_i;\quad \varepsilon_i\varepsilon_{i+1}p_i=\delta'\varepsilon_i p_i.
\end{gather*}
\end{proposition}

Using Proposition~\ref{prop115} we obtain the following:

\begin{theorem}[Schur-Weyl duality]
\label{thm114}{\tiny\hspace{1mm}}

\begin{enumerate}[$($i$)$]
\item\label{thm114.1} The image of $\mathcal{O}_k$ in $\mathrm{End}_{\mathbb{C}}(\overline{V}^{\otimes m})$
coincides with the centralizer of $R_m$.
\item\label{thm114.2} Mapping 
\begin{displaymath}
\begin{array}{lcll}
\gamma&\mapsto&\gamma,& \gamma\in S_m;\\
u_{m,i}&\mapsto& \varepsilon_i,& i=1,2,\dots,m;\\
p_i&\mapsto&\zeta_{i},& i=1,2,\dots,m-1;
\end{array}
\end{displaymath}
defines an epimorphism $\Psi:\mathfrak{R}_m(k+1,1)\tto R_m$.
\item\label{thm114.3} For $m\leq k$ the epimorphism $\Psi$ is injective and hence bijective. 
\end{enumerate}
\end{theorem}

\begin{remark}\label{rem325}
{\rm 
The paper \cite{Gr} mentions a  Schur-Weyl duality for the rook Brauer algebra attributing it to an unpublished
paper of G.~Benkart and T.~Halverson, but that paper was never made available, \cite{Ha}. We were recently 
informed by T.~Halverson that a  Schur-Weyl duality for the rook Brauer algebra is now in process of 
independently being worked out by one of his students (\cite{De}).
} 
\end{remark}

\begin{proof}
As $\overline{V}^{\otimes m}$ is a semi-simple representation of $\mathcal{O}_k$, claim \eqref{thm114.1}
follows from Proposition~\ref{prop111} by standard arguments (as e.g. in the proof of \cite[Theorem~4.2.10]{GW}). 
To prove claim \eqref{thm114.2} we just have to check the relations
given by Proposition~\ref{prop115}. The first three lines of relations follow from \cite{Brauer37},
the fourth line follows from \cite{So}. To check the remaining two lines is a straightforward computation.

Because of claim \eqref{thm114.2}, to prove claim \eqref{thm114.3} it is enough to prove the equality 
$\dim(R_m)=\dim(\mathfrak{R}_m(k+1,1))$ for $m\leq k$. Denote by $\mathbf{N}$ the principal 
minor of the matrix $\mathbf{M}+\mathbf{M}^t+\mathbbm{1}$ given by rows and columns with indices in $\Lambda$. 
Set $L:=\oplus_{\lambda\in\Lambda}L_{\lambda}$ and let $\mathcal{C}:=\mathrm{add}(L)$. The endofunctor 
$\overline{V}\otimes{}_{-}$ of $\mathcal{C}$ is exact and hence induces a linear transformation
of the Grothendieck group $[\mathcal{C}]$. Denote by $N$ the matrix of this transformation in the basis
of simple modules. Comparing \eqref{eq112} with Theorem~\ref{thm24} and Corollary~\ref{cor25} we get 
$N=\mathbf{N}$. This yields $\dim(R_m)=\dim(\mathfrak{R}_m(k+1,1))$ and thus $R_m\cong \mathfrak{R}_m(k+1,1)$, 
completing the proof.
\end{proof}

\section{Theorem~\ref{thm1} in the degenerate cases}\label{s4}

\subsection{The case $\delta'=0$}\label{s4.1}

In this subsection we assume $\Bbbk=\mathbb{C}$ and $\delta'=0$. Similarly to Subsection~\ref{s3.1}, 
for every left cell $\mathcal{L}$ of $\mathtt{R}_n$ we have the corresponding cell module $\mathbf{C}_{\mathcal{L}}$.
If $k$ is the propagating number of $\mathcal{L}$, then the group $S_k$ acts freely on $\mathbf{C}_{\mathcal{L}}$ by
automorphisms permuting the right ends of the propagating lines in diagrams from $\mathcal{L}$. 

\begin{lemma}\label{lem501}
Let $\mathcal{L}$ and $\mathcal{L}'$ be two left cells having the same propagating number $k$. Then 
$\mathbf{C}_{\mathcal{L}}\cong \mathbf{C}_{\mathcal{L}'}$ as $\mathfrak{R}_n$--$S_k$-bimodules.
\end{lemma}

\begin{proof}
For every $d\in \mathcal{L}$ there is a unique $d'\in \mathcal{L}'$ having the same propagating lines
and the same parts contained in $\underline{n}$ as $d$. From the definition of multiplication in 
$\mathfrak{R}_n$ it follows that mapping $d\to d'$ defines the required isomorphism.
\end{proof}

Similarly to Subsection~\ref{s3.1}, for $\lambda\vdash k$ we can now define the $\mathfrak{R}_n$-module
\begin{displaymath}
\Delta(\lambda):= \mathbf{C}_{\mathcal{L}}\otimes_{\mathbb{C}[S_k]}\mathcal{S}(\lambda),
\end{displaymath}
where $\mathcal{S}(\lambda)$ is the Specht module for $S_k$ corresponding to $\lambda$. From 
\cite{GX} it follows that the algebra $\mathfrak{R}_n$ is cellular in the sense of \cite{GL} with $\Delta(\lambda)$
being the corresponding structural modules (usually also called {\em cell modules}). In particular, from the
general theory of cellular algebras, see for example \cite[Theorem~3.7]{GL}, we have that the Cartan matrix of 
$\mathfrak{R}_n$ equals $MM^t$, where $M$ is the decomposition matrix for $\Delta$'s. In this subsection we will 
determine the latter. But first let us determine simple $\mathfrak{R}_n$-modules. 

\begin{proposition}\label{lem502}
Let $I$ be the linear span of all  $d\in \mathtt{R}_n\setminus \mathtt{B}_n$.
\begin{enumerate}[$($i$)$]
\item\label{lem502.1} The set $I$ is a nilpotent ideal of $\mathfrak{R}_n$.
\item\label{lem502.2} We have $\mathfrak{R}_n/I\cong \mathfrak{B}_n(\delta)$. 
\end{enumerate}
\end{proposition}

\begin{proof}
If $d$ is a diagram containing a singleton, then, composing $d$ with any diagram from any side, the singleton
of $d$ either contributes to a singleton of the product or to an open string, removing which we get $0$ as
$\delta'=0$. This implies that $I$ is a two-sided ideal of $\mathfrak{R}_n$. Given $n+1$ diagrams with
singletons we have at least $2n+2$ singletons. At most $2n$ of these can contribute to singletons in the
product. Hence at least one of these singletons must contribute to an open string in the product which implies
that the product is zero. Thus $I^{n+1}=0$ and claim \eqref{lem502.1} follows (note that it is easy to see 
that $I^{n}\neq 0$).
 
Let $d\in \mathtt{B}_n$. Then, mapping the class of $d$ in $\mathfrak{R}_n/I$ to $d$, defines an isomorphism 
$\mathfrak{R}_n/I\to\mathfrak{B}_n(\delta)$ and completes the proof.
\end{proof}

In particular, from Proposition~\ref{lem502} it follows that the algebras $\mathfrak{R}_n$ and 
$\mathfrak{B}_n=\mathfrak{B}_n(\delta)$ have the same simple modules. For $k\in\{n,n-2,n-4,\dots,1/0\}$
(here $1/0$ means `$1$ or $0$ depending on the parity of $n$')
and $\lambda\vdash k$ denote by $\Delta^{\mathfrak{B}}(\lambda)$ the {\em Brauer Specht module} for the algebra 
$\mathfrak{B}_n$ associated with $\lambda$, constructed similarly to the above (but using only diagrams from
$\mathtt{B}_n$). The decomposition matrix for these modules is completely determined in \cite{Martin08-9}. 
Below, for every module $\Delta(\lambda)$ we construct a filtration with subquotients isomorphic to Brauer
Specht modules and explicitly determine the multiplicities of Brauer Specht modules in these filtrations,
thus determining the decomposition matrix for $\Delta$'s. For a left cell $\overline{\mathcal{L}}$ in
$\mathtt{B}_n$ we denote by $\mathbf{C}^{\mathfrak{B}}_{\overline{\mathcal{L}}}$ the corresponding
cell module for $\mathfrak{B}_n$ constructed similarly to the above.

Let $\mathcal{L}$ be a left cell with propagating number $k$. For $l\in\{n-k,n-k-2,\dots,1/0\}$ denote by
$\mathcal{L}_l$ the set of all diagrams $d\in \mathcal{L}$ with exactly $l$ singletons in $\underline{n}$.
Then $\mathcal{L}$ is a disjoint union of the $\mathcal{L}_l$'s. Further, for each $l$ as above let
$\mathbf{C}_{\mathcal{L}}^{(l)}$ denote the linear span (inside $\mathbf{C}_{\mathcal{L}}$) of all
diagrams $d\in \mathcal{L}$ with at least $l$ singletons in $\underline{n}$. Obviously,
$\mathbf{C}_{\mathcal{L}}^{(1/0)}=\mathbf{C}_{\mathcal{L}}$.

\begin{lemma}\label{lem503}
Let $l\in\{n-k,n-k-2,\dots,1/0\}$.
\begin{enumerate}[$($i$)$]
\item\label{thm503.1} Each $\mathbf{C}_{\mathcal{L}}^{(l)}$ is a submodule of $\mathbf{C}_{\mathcal{L}}$
contained in $\mathbf{C}_{\mathcal{L}}^{(l-2)}$ (the latter for $l\neq 1/0$).
\item\label{thm503.2} The set $\mathcal{L}_l$ descends to a basis in the 
quotient $\mathbf{C}_{\mathcal{L}}^{(l)}/\mathbf{C}_{\mathcal{L}}^{(l+2)}$
(where $\mathbf{C}_{\mathcal{L}}^{(n-k+2)}:=0$).
\item\label{thm503.3} The free action of $S_k$ on $\mathcal{L}$ preserves
each $\mathcal{L}_l$ and hence induces a free action of $S_k$ on the quotient
$\mathbf{C}_{\mathcal{L}}^{(l)}/\mathbf{C}_{\mathcal{L}}^{(l+2)}$.
\end{enumerate}
\end{lemma}

\begin{proof} 
Let $d$ be a diagram with a singleton and $d'$ be any diagram. If the propagating number of $d'd$ is the same
as the propagating number of $d$, then each singleton in $d$ lying in $\underline{n}$ either contributes to a 
singleton in $d'd$ lying in $\underline{n}$ or to an open string making the product $d'd$ zero. This means
that the number of singleton in $\underline{n}$ can only increase in the product. Claim \eqref{thm503.1} follows.
Claims \eqref{thm503.2} and \eqref{thm503.3} are straightforward.
\end{proof}

Since simple $\mathfrak{R}_n$ and $\mathfrak{B}_n$ modules coincide, it is left to examine the 
$\mathfrak{B}_n$-module structure of the sections 
\begin{displaymath}
N_l(\lambda):= \mathbf{C}_{\mathcal{L}}^{(l)}/\mathbf{C}_{\mathcal{L}}^{(l+2)}\otimes_{\mathbb{C}[S_k]} 
\mathcal{S}(\lambda),
\end{displaymath}
for all $\lambda\vdash k$ and $l\in\{n-k,n-k-2,\dots,1/0\}$. These sections are well-defined by Lemma~\ref{lem503}.

For $\lambda\vdash k$ and $l\in\{n-k,n-k-2,\dots,1/0\}$ denote by $\mathcal{X}_{\lambda,l}$ the set of
all $\mu\vdash k+l$ for which $\mathcal{S}(\mu)$ occurs as a summand in the induced module 
\begin{displaymath}
\mathrm{Ind}_{S_k\oplus S_l}^{S_{k+l}}(\mathcal{S}(\lambda)\otimes \mathcal{S}((l))) 
\end{displaymath}
and let $m_{\mu}$ denote the corresponding multiplicity (this multiplicity can be computed using the 
classical Littlewood-Richardson rule, see \cite[4.9]{Sa}). Note that the Specht module $\mathcal{S}((l))$ is
the trivial $S_l$-module. Similarly, denote by $\mathcal{Y}_{\lambda,l}$ the set of
all $\mu\vdash k+l$ for which $\mathcal{S}(\mu)$ occurs as a summand in the induced module 
\begin{displaymath}
\mathrm{Ind}_{S_k\oplus S_l}^{S_{k+l}}(\mathbb{C}[S_k]\otimes \mathcal{S}((l))) 
\end{displaymath}
and let $m'_{\mu}$ denote the corresponding multiplicity. Our final step is the following:

\begin{theorem}\label{thm505}
The  $\mathfrak{B}_n$-module $N_l(\lambda)$ decomposes into a direct sum  as follows:
\begin{displaymath}
N_l(\lambda)\cong\bigoplus_{\mu\in \mathcal{X}_{\lambda,l}} m_{\mu} \Delta^{\mathfrak{B}}(\mu).
\end{displaymath}
\end{theorem}

\begin{proof}
Taking into account the action of $S_k$, it is of course enough to show that we have the following 
isomorphism of $\mathfrak{B}_n$-modules:
\begin{displaymath}
\mathbf{C}_{\mathcal{L}}^{(l)}/\mathbf{C}_{\mathcal{L}}^{(l+2)}\cong 
\bigoplus_{\mu\in \mathcal{Y}_{\lambda,l}} m'_{\mu} \Delta^{\mathfrak{B}}(\mu). 
\end{displaymath}
The left hand side has basis $\mathcal{L}_l$ by Proposition~\ref{lem503}\eqref{thm503.2}. 
If $d\in \mathcal{L}_l$ and $d'\in\mathtt{B}_n$, then $d'd\neq0$ implies that every $\underline{n}$-singleton in $d$
corresponds (maybe via some connection on the equator) to a propagating line in $d'$. In particular, the 
propagating number of $d'$ must be at least $k+l$. Note also that $\underline{n}$-singletons in $d$ are
indistinguishable under interchange. 

Without loss of generality we may assume that $d$ has propagating lines $\{1,1'\}$, $\{2,2'\},\dots$, $\{k,k'\}$
and $\underline{n}$-singletons $\{k+1\},\{k+2\},\dots,\{k+l\}$. Let $\overline{\mathcal{L}}$ be a left cell 
in $\mathtt{B}_n$ with propagating number $k+l$. Without loss of  generality we may assume that 
$\overline{\mathcal{L}}$ contains a diagram with propagating lines $\{1,1'\},\{2,2'\},\dots,\{k+l,(k+l)'\}$.
Define the linear map 
\begin{displaymath}
\varphi:\mathbf{C}_{\mathcal{L}}^{(l)}/\mathbf{C}_{\mathcal{L}}^{(l+2)}\to 
\mathbf{C}^{\mathfrak{B}}_{\overline{\mathcal{L}}} 
\end{displaymath}
on the basis element $d\in \mathcal{L}_l$ as follows: $\varphi(d)$ is the sum of all 
$\overline{d}\in \overline{\mathcal{L}}$ such that $d$ and $\overline{d}$ have the same propagating
lines containing an element in $\{1',2',\dots,k'\}$ and, moreover, $d$ and $\overline{d}$
have the same pair parts in $\underline{n}$ (see an example on Figure~\ref{fig507}). 
It is now easy to check that this map provides the required isomorphism. 
\end{proof}

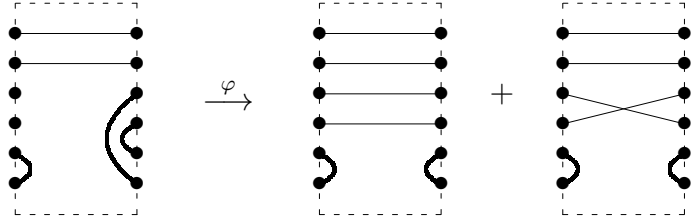
\begin{figure}
\special{em:linewidth 0.4pt} \unitlength 0.80mm
\begin{picture}(115,45)
\put(05,10){\makebox(0,0)[cc]{$\bullet$}}
\put(05,15){\makebox(0,0)[cc]{$\bullet$}}
\put(05,20){\makebox(0,0)[cc]{$\bullet$}}
\put(05,25){\makebox(0,0)[cc]{$\bullet$}}
\put(05,30){\makebox(0,0)[cc]{$\bullet$}}
\put(05,35){\makebox(0,0)[cc]{$\bullet$}}
\put(25,10){\makebox(0,0)[cc]{$\bullet$}}
\put(25,15){\makebox(0,0)[cc]{$\bullet$}}
\put(25,20){\makebox(0,0)[cc]{$\bullet$}}
\put(25,25){\makebox(0,0)[cc]{$\bullet$}}
\put(25,30){\makebox(0,0)[cc]{$\bullet$}}
\put(25,35){\makebox(0,0)[cc]{$\bullet$}}
\put(55,10){\makebox(0,0)[cc]{$\bullet$}}
\put(55,15){\makebox(0,0)[cc]{$\bullet$}}
\put(55,20){\makebox(0,0)[cc]{$\bullet$}}
\put(55,25){\makebox(0,0)[cc]{$\bullet$}}
\put(55,30){\makebox(0,0)[cc]{$\bullet$}}
\put(55,35){\makebox(0,0)[cc]{$\bullet$}}
\put(75,10){\makebox(0,0)[cc]{$\bullet$}}
\put(75,15){\makebox(0,0)[cc]{$\bullet$}}
\put(75,20){\makebox(0,0)[cc]{$\bullet$}}
\put(75,25){\makebox(0,0)[cc]{$\bullet$}}
\put(75,30){\makebox(0,0)[cc]{$\bullet$}}
\put(75,35){\makebox(0,0)[cc]{$\bullet$}}
\put(95,10){\makebox(0,0)[cc]{$\bullet$}}
\put(95,15){\makebox(0,0)[cc]{$\bullet$}}
\put(95,20){\makebox(0,0)[cc]{$\bullet$}}
\put(95,25){\makebox(0,0)[cc]{$\bullet$}}
\put(95,30){\makebox(0,0)[cc]{$\bullet$}}
\put(95,35){\makebox(0,0)[cc]{$\bullet$}}
\put(115,10){\makebox(0,0)[cc]{$\bullet$}}
\put(115,15){\makebox(0,0)[cc]{$\bullet$}}
\put(115,20){\makebox(0,0)[cc]{$\bullet$}}
\put(115,25){\makebox(0,0)[cc]{$\bullet$}}
\put(115,30){\makebox(0,0)[cc]{$\bullet$}}
\put(115,35){\makebox(0,0)[cc]{$\bullet$}}
\dashline{1}(05,05)(05,40)
\dashline{1}(25,40)(05,40)
\dashline{1}(25,40)(25,05)
\dashline{1}(05,05)(25,05)
\dashline{1}(55,05)(55,40)
\dashline{1}(75,40)(55,40)
\dashline{1}(75,40)(75,05)
\dashline{1}(55,05)(75,05)
\dashline{1}(95,05)(95,40)
\dashline{1}(115,40)(95,40)
\dashline{1}(115,40)(115,05)
\dashline{1}(95,05)(115,05)
\drawline(05,35)(25,35)
\drawline(05,30)(25,30)
\drawline(55,35)(75,35)
\drawline(55,30)(75,30)
\drawline(95,35)(115,35)
\drawline(95,30)(115,30)
\drawline(55,25)(75,25)
\drawline(55,20)(75,20)
\drawline(95,25)(115,20)
\drawline(95,20)(115,25)
\linethickness{1pt}
\qbezier(05,10)(10,12.50)(05,15)
\qbezier(55,10)(60,12.50)(55,15)
\qbezier(95,10)(100,12.50)(95,15)
\qbezier(75,10)(70,12.50)(75,15)
\qbezier(115,10)(110,12.50)(115,15)
\qbezier(25,10)(15,17.50)(25,25)
\qbezier(25,15)(20,17.50)(25,20)
\put(85,25){\makebox(0,0)[cc]{$+$}}
\put(40,25){\makebox(0,0)[cc]{$\overset{\varphi}{\longrightarrow}$}}
\end{picture}
\caption{The map $\varphi$ in Theorem~\ref{thm505}}\label{fig507} 
\end{figure}

\subsection{The case $\delta=1$ and $\delta'\neq 0$}\label{s4.2}

The algebra $\mathfrak{R}_n(1,0)$ was considered in the previous subsection and hence here we assume
$\delta'\neq 0$. From Proposition~\ref{prop115} it follows that, mapping $s_i\mapsto s_i$, $p_i\mapsto p_i$ and 
$\varepsilon_i\mapsto \frac{1}{\delta'}\varepsilon_i$ for all $i$ for which the expression makes sense, extends
to an isomorphism $\mathfrak{R}_n(1,\delta')\cong \mathfrak{R}_n(1,1)$ and allows us to concentrate here on the
case $\delta=\delta'=1$ (that is on the case of the usual complex semigroup algebra of the partial Brauer semigroup
from \cite{Mazorchuk95}). 

From Proposition~\ref{prop2} we have $\mathfrak{R}_n=\mathfrak{R}_n^0\oplus\mathfrak{R}_n^1$. From
Section~\ref{s2} we have the functors $\mathrm{F}$ and $\mathrm{F}_{\omega}$ given by \eqref{eqf1} and \eqref{eqf2},
respectively. To continue, we first have to refine some results from Section~\ref{s2}:

\begin{proposition}\label{prop521}
Under the assumptions $\delta=\delta'=1$ we have:

\begin{enumerate}[$($i$)$]
\item\label{prop521.1} For odd $n$ the functor $\mathrm{F}$ induces an equivalence
$\mathfrak{R}_n^0\text{-}\mathrm{mod}\cong \mathfrak{B}_n(\delta-1)\text{-}\mathrm{mod}$.
\item\label{prop521.2} For even $n$ the functor $\mathrm{F}_{\omega}$ induces an equivalence
$\mathfrak{R}_n^1\text{-}\mathrm{mod}\cong \mathfrak{B}_{n-1}(\delta-1)\text{-}\mathrm{mod}$.
\end{enumerate}
\end{proposition}

\begin{proof}
We need only to modify the proofs of Propositions~\ref{prop4} and \ref{prop6} to show that 
in the case $n-2k\neq 0$ or $n-2k-1\neq 0$, respectively, the corresponding elements
$\omega_{\{1,2,\dots,2k\}}$ and $\omega_{\{1,2,\dots,2k+1\}}$ belong to $\mathfrak{R}_n\hat{u}_n\mathfrak{R}_n$
and $\mathfrak{R}_n{\omega}\mathfrak{R}_n$, respectively. How this can be done is illustrated
by Figure~\ref{fig522}, we leave the details to the reader.
\end{proof}

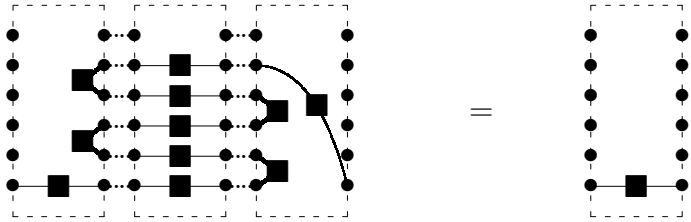
\begin{figure}
\special{em:linewidth 0.4pt} \unitlength 0.80mm
\begin{picture}(125,45)
\put(05,10){\makebox(0,0)[cc]{$\bullet$}}
\put(05,15){\makebox(0,0)[cc]{$\bullet$}}
\put(05,20){\makebox(0,0)[cc]{$\bullet$}}
\put(05,25){\makebox(0,0)[cc]{$\bullet$}}
\put(05,30){\makebox(0,0)[cc]{$\bullet$}}
\put(05,35){\makebox(0,0)[cc]{$\bullet$}}
\put(20,10){\makebox(0,0)[cc]{$\bullet$}}
\put(20,15){\makebox(0,0)[cc]{$\bullet$}}
\put(20,20){\makebox(0,0)[cc]{$\bullet$}}
\put(20,25){\makebox(0,0)[cc]{$\bullet$}}
\put(20,30){\makebox(0,0)[cc]{$\bullet$}}
\put(20,35){\makebox(0,0)[cc]{$\bullet$}}
\put(25,10){\makebox(0,0)[cc]{$\bullet$}}
\put(25,15){\makebox(0,0)[cc]{$\bullet$}}
\put(25,20){\makebox(0,0)[cc]{$\bullet$}}
\put(25,25){\makebox(0,0)[cc]{$\bullet$}}
\put(25,30){\makebox(0,0)[cc]{$\bullet$}}
\put(25,35){\makebox(0,0)[cc]{$\bullet$}}
\put(40,10){\makebox(0,0)[cc]{$\bullet$}}
\put(40,15){\makebox(0,0)[cc]{$\bullet$}}
\put(40,20){\makebox(0,0)[cc]{$\bullet$}}
\put(40,25){\makebox(0,0)[cc]{$\bullet$}}
\put(40,30){\makebox(0,0)[cc]{$\bullet$}}
\put(40,35){\makebox(0,0)[cc]{$\bullet$}}
\put(45,10){\makebox(0,0)[cc]{$\bullet$}}
\put(45,15){\makebox(0,0)[cc]{$\bullet$}}
\put(45,20){\makebox(0,0)[cc]{$\bullet$}}
\put(45,25){\makebox(0,0)[cc]{$\bullet$}}
\put(45,30){\makebox(0,0)[cc]{$\bullet$}}
\put(45,35){\makebox(0,0)[cc]{$\bullet$}}
\put(60,10){\makebox(0,0)[cc]{$\bullet$}}
\put(60,15){\makebox(0,0)[cc]{$\bullet$}}
\put(60,20){\makebox(0,0)[cc]{$\bullet$}}
\put(60,25){\makebox(0,0)[cc]{$\bullet$}}
\put(60,30){\makebox(0,0)[cc]{$\bullet$}}
\put(60,35){\makebox(0,0)[cc]{$\bullet$}}
\put(100,10){\makebox(0,0)[cc]{$\bullet$}}
\put(100,15){\makebox(0,0)[cc]{$\bullet$}}
\put(100,20){\makebox(0,0)[cc]{$\bullet$}}
\put(100,25){\makebox(0,0)[cc]{$\bullet$}}
\put(100,30){\makebox(0,0)[cc]{$\bullet$}}
\put(100,35){\makebox(0,0)[cc]{$\bullet$}}
\put(115,10){\makebox(0,0)[cc]{$\bullet$}}
\put(115,15){\makebox(0,0)[cc]{$\bullet$}}
\put(115,20){\makebox(0,0)[cc]{$\bullet$}}
\put(115,25){\makebox(0,0)[cc]{$\bullet$}}
\put(115,30){\makebox(0,0)[cc]{$\bullet$}}
\put(115,35){\makebox(0,0)[cc]{$\bullet$}}
\dashline{1}(05,05)(05,40)
\dashline{1}(20,40)(05,40)
\dashline{1}(20,40)(20,05)
\dashline{1}(05,05)(20,05)
\dashline{1}(25,05)(25,40)
\dashline{1}(40,40)(25,40)
\dashline{1}(40,40)(40,05)
\dashline{1}(25,05)(40,05)
\dashline{1}(45,05)(45,40)
\dashline{1}(60,40)(45,40)
\dashline{1}(60,40)(60,05)
\dashline{1}(45,05)(60,05)
\dashline{1}(100,05)(100,40)
\dashline{1}(115,40)(100,40)
\dashline{1}(115,40)(115,05)
\dashline{1}(100,05)(115,05)
\dottedline[.]{1}(20.00,10.00)(25.00,10.00)
\dottedline[.]{1}(20.00,15.00)(25.00,15.00)
\dottedline[.]{1}(20.00,20.00)(25.00,20.00)
\dottedline[.]{1}(20.00,25.00)(25.00,25.00)
\dottedline[.]{1}(20.00,30.00)(25.00,30.00)
\dottedline[.]{1}(20.00,35.00)(25.00,35.00)
\dottedline[.]{1}(40.00,10.00)(45.00,10.00)
\dottedline[.]{1}(40.00,15.00)(45.00,15.00)
\dottedline[.]{1}(40.00,20.00)(45.00,20.00)
\dottedline[.]{1}(40.00,25.00)(45.00,25.00)
\dottedline[.]{1}(40.00,30.00)(45.00,30.00)
\dottedline[.]{1}(40.00,35.00)(45.00,35.00)
\drawline(05,10)(20,10)
\drawline(25,10)(40,10)
\qbezier(45,30)(55,30)(60,10)
\drawline(100,10)(115,10)
\drawline(25,15)(40,15)
\drawline(25,20)(40,20)
\drawline(25,25)(40,25)
\drawline(25,30)(40,30)
\linethickness{1pt}
\qbezier(20,15)(15,17.50)(20,20)
\qbezier(20,25)(15,27.50)(20,30)
\qbezier(45,15)(50,12.50)(45,10)
\qbezier(45,25)(50,22.50)(45,20)
\put(82,22){\makebox(0,0)[cc]{$=$}}
\put(12.50,10){\makebox(0,0)[cc]{$\blacksquare$}}
\put(32.50,10){\makebox(0,0)[cc]{$\blacksquare$}}
\put(55.00,23.50){\makebox(0,0)[cc]{$\blacksquare$}}
\put(107.50,10){\makebox(0,0)[cc]{$\blacksquare$}}
\put(32.50,15){\makebox(0,0)[cc]{$\blacksquare$}}
\put(32.50,20){\makebox(0,0)[cc]{$\blacksquare$}}
\put(32.50,25){\makebox(0,0)[cc]{$\blacksquare$}}
\put(32.50,30){\makebox(0,0)[cc]{$\blacksquare$}}
\put(16.50,17.50){\makebox(0,0)[cc]{$\blacksquare$}}
\put(16.50,27.50){\makebox(0,0)[cc]{$\blacksquare$}}
\put(48.50,12.50){\makebox(0,0)[cc]{$\blacksquare$}}
\put(48.50,22.50){\makebox(0,0)[cc]{$\blacksquare$}}
\end{picture}
\caption{Illustration to the proof of Proposition~\ref{prop521}}\label{fig522} 
\end{figure}

For even $n$ the functor $\mathrm{F}$ does not induce an equivalence
between 
$\mathfrak{R}_n^0\text{-}\mathrm{mod}$ 
and $\mathfrak{B}_n(0)\text{-}\mathrm{mod}$. In fact, the latter two categories are not equivalent as
the first one has one extra simple module, corresponding to the empty
partition of $0$. 
(Corresponding to the striking result above, that $\mathfrak{R}_n$ remains
quasihereditary for all $\delta$, while $\mathfrak{B}_n$ does not.)
The functor 
$\mathrm{F}$ is still full, faithful and dense on $\mathrm{add}(\mathfrak{R}_n\hat{u}_n)$ and hence
for any indecomposable direct summands $P$ and $Q$ of $\mathfrak{R}_n\hat{u}_n$ we have
$\mathrm{Hom}_{\mathfrak{R}_n} (P,Q)\cong \mathrm{Hom}_{\mathfrak{B}_n(0)}(\mathrm{F}\,P,\mathrm{F}\,Q)$,
the latter being known by \cite{Martin08-9}. This determines a maximal square submatrix of the Cartan matrix for 
$\mathfrak{R}_n$. The remaining row and column are the ones corresponding to the projective cover $P_{\varnothing}$ 
of the one-dimensional trivial module (the module on which all diagrams act as the identity) since 
$\hat{u}_n$ obviously annihilates the simple top of this module. 

Let $\mathcal{L}$ denote the left cell of the diagram $u_{\underline{n}}$ consisting only of singletons. 
Then $P_{\varnothing}\cong \mathbf{C}_{\mathcal{L}}$. As $u_{\underline{n}}d u_{\underline{n}}=u_{\underline{n}}$
for any $d\in\mathtt{R}_n$, it follows that $\mathbf{C}_{\mathcal{L}}$ has $1$-dimensional endomorphism
ring implying that the corresponding diagonal  entry in the Cartan matrix is $1$. 

Note that for $\delta=\delta'=1$ our algebra $\mathfrak{R}_n$ is the semigroup algebra of the partial Brauer
semigroup (see \cite{Mazorchuk95}). The involution $\star$ stabilizes all idempotents $u_X$, $X\subset\underline{n}$, 
and induces the usual involution (taking the inverse) on the
corresponding maximal subgroups. 
General 
semigroup representation theory (see \cite{CP,GM,GMS}) then says that $\star$ 
induces a simple preserving contravariant self-equivalence on 
$\mathfrak{R}_n\text{-}\mathrm{mod}$, which implies that the Cartan matrix of $\mathfrak{R}_n$
is symmetric. Thus it remains to determine the composition multiplicities of all other simple modules in 
$\mathbf{C}_{\mathcal{L}}$. As the functor $\mathrm{F}$ sends Specht modules for $\mathfrak{R}_n$ to
Specht module for $\mathfrak{B}_n(0)$, by Proposition~\ref{prop12} it equates 
these composition multiplicities to the corresponding
composition multiplicities for $\mathfrak{B}_n(0)$ (even though the corresponding Specht module is
{\em not} projective there). The latter are known by \cite{Martin08-9}, using Brauer
reciprocity.

For odd $n$ we have exactly the same situation with the functor $\mathrm{F}_{\omega}$. This discussion 
completely determines the Cartan matrix of $\mathfrak{R}_n$.

\subsection{The Temperley-Lieb algebra and its partialization}\label{s4.3}

All constructions and results of the present paper transfer mutatis mutandis to the case of the (partial)
Temperley-Lieb algebra, that is a subalgebra of the (partial) Brauer algebra with basis given by planar
diagrams. Some further related algebras can be found in the recent preprint \cite{BH}.

\vspace{0.3cm}

\noindent
P.~M.: Department of Pure Mathematics, University of Leeds, Leeds, 
LS2 9JT, UK, e-mail: {\tt ppmartin\symbol{64}maths.leeds.ac.uk }
\vspace{0.3cm}

\noindent
V.~M: Department of Mathematics, Uppsala University, Box. 480,
SE-75106, Uppsala, SWEDEN, email: {\tt mazor\symbol{64}math.uu.se}

\end{document}